\documentclass[11pt]{articlefederico}

\usepackage{fullpage}

\usepackage{todonotes,xypic}
\usepackage{amsthm}
\usepackage{amsmath}
\usepackage{amssymb}
\usepackage{color}
\usepackage{enumerate}
\usepackage{bbm}

\usepackage{graphicx}

\newtheorem{theorem}{Theorem}[section]
\newtheorem{lemma}[theorem]{Lemma}
\newtheorem{remark}[theorem]{Remark}
\newtheorem{question}[theorem]{Question}

\newtheorem{definition}[theorem]{Definition}
\newtheorem{proposition}[theorem]{Proposition}

\newtheorem{example}[theorem]{Example}

\makeatletter
\newtheorem*{rep@theorem}{\rep@title}
\newcommand{\newreptheorem}[2]{%
\newenvironment{rep#1}[1]{%
 \def\rep@title{#2 \ref{##1}}%
 \begin{rep@theorem}}%
 {\end{rep@theorem}}}
\makeatother

\newreptheorem{theorem}{Theorem}
\newreptheorem{corollary}{Corollary}

\renewcommand{\AA}{\mathbb{A}}

\newcommand{\kk}{{\mathbbm{k}}}
\renewcommand{\k}{{\mathbbm{k}}}

\renewcommand{\a}{{\textbf{a}}}

\newcommand{\B}{{\cal B}}
\newcommand{\C}{{\cal C}}
\newcommand{\D}{{\cal D}}

\newcommand{\II}{{\mathcal{I}}}

\renewcommand{\L}{{\widetilde{L}}}

\newcommand{\PP}{\mathbb{P}}

\newcommand{\RR}{\mathbb{R}}

\newcommand{\X}{{\widetilde{X}}}
\newcommand{\ZZ}{\mathbb{Z}}

\newcommand{\Tor}{\mathrm{Tor}}

\newcommand{\conv}{\mathrm{conv }}
\newcommand{\supp}{\mathrm{supp }}

\newcommand{\bideg}{\mathrm{bideg \,\, }}
\newcommand{\multideg}{\mathrm{mdeg \,\, }}

\usepackage{pgf}
\usepackage{tikz}
\usetikzlibrary{arrows,automata}

\newcommand{\blue}[1]{#1}
\newcommand{\bblue}[1]{{\color{blue} #1}}

\DeclareMathOperator{\initial}{in}
\DeclareMathOperator{\codim}{codim}

\title{\textsf{The closure of a linear space in a product of lines}}
\author{\textsf{Federico Ardila\footnote{\textsf{San Francisco State University, San Francisco, USA; Universidad de Los Andes, Bogot\'a, Colombia. federico@sfsu.edu}} \qquad Adam Boocher\footnote{\textsf{University of Edinburgh, Edinburgh, UK.  adam.boocher@ed.ac.uk \newline Ardila was partially supported by the US National Science Foundation CAREER Award DMS-0956178 and, the SFSU-Colombia Combinatorics Initiative. Boocher was partially supported by an NSF Graduate Research Fellowship.}
}}}
\date{}
\begin{document}
\maketitle

\begin{abstract}
Given a linear space $L$ in affine space $\mathbb{A}^n$, we study its closure $\widetilde{L}$ in the product of projective lines $(\mathbb{P}^1)^n$. We show that the degree, multigraded Betti numbers, defining equations, and universal Gr\"obner basis of its defining ideal $I(\widetilde{L})$ are all combinatorially determined by the matroid $M$ of $L$.  
We also prove $I(\widetilde{L})$ and all of its initial ideals are Cohen-Macaulay with the same Betti numbers, and can be used to compute the $h$-vector of $M$.
This variety $\widetilde{L}$ also gives rise to two new objects with interesting properties: the \emph{cocircuit polytope} and the \emph{external activity complex} of a matroid.
%
\end{abstract}

\section{\textsf{Introduction.}}
If $L\subset \mathbb{A}^n$ is a $d$-dimensional linear space in affine space $\mathbb{A}^n$ over an infinite field $\kk$,  its usual closure 
in $\mathbb{P}^n$ is one of the simplest projective varieties. 
It  is trivially a projective linear space, and its defining ideal 
is generated by $n-d$ linear forms.
However, this
is only one of many possible closures!  

In this paper we study the next simplest possibility. Choose a frame $F =\{\langle e_1\rangle, \ldots, \langle e_n\rangle \}$ where the $e_i$ form a basis of $n$-space and $\langle\  \rangle$ denotes linear span. This frame gives rise to an embedding $\mathbb{A}^n \hookrightarrow (\PP^1)^n$, and we consider the closure $\widetilde{L} \subset (\PP^1)^n$ of $L$ in this product of projective lines. This case is already quite interesting; several algebraic, combinatorial, and geometric invariants of $\widetilde{L}$ are determined purely combinatorially. There is a matroid $M$ which encodes the relative position of $L$ with respect to the frame $F$. Our main result is that this matroid, which in principle only knows linear information about $L$, actually determines much of the structure of $\widetilde{L}$:

\begin{theorem}
Let $L\subset \mathbb{A}^n$ be a linear space and let $\widetilde{L}$ be its closure via the embedding $\mathbb{A}^n \hookrightarrow (\mathbb{P}^1)^n$.
The following invariants depend only on the matroid of $L$:  the $\mathbb{Z}^n$--multidegree of $\widetilde{L}$, the multigraded Betti numbers of $I(\widetilde{L})$ and all its initial ideals, the number of minimal generators of the defining ideal $I(\widetilde{L})$, and the set of initial ideals of $I(\widetilde{L})$.  Furthermore, $I(\widetilde{L})$ and all of its initial ideals are Cohen-Macaulay with the same Betti numbers. 
\end{theorem}

In fact, when we state this result more precisely in Theorem \ref{longthm cases}, we will see that several important matroid invariants are realized as algebro-geometric invariants of the projective variety $\widetilde{L}$. For instance, we can use $\widetilde{L}$ to compute algebraically the $h$-vector of the matroid of $L$, and the internal activities of its bases under any order. 
This variety also gives rise to two new objects with interesting properties: the \emph{cocircuit polytope} and the \emph{external activity complex} of a matroid.

%

The paper is organized as follows. In Section \ref{sec:closures} we define our main subject of study: the closure of a linear space $L \subset \AA^n$ in $(\PP^1)^n$. We state our main algebraic and combinatorial theorems in Sections \ref{sec:results} and Section \ref{sec:matroidresults} respectively, illustrating them in an example.  Section \ref{sec:related} discusses related work. In Sections \ref{sec:prelimcombin} and \ref{sec:prelimalg} we collect the basic facts from matroid theory and commutative algebra that we will need. In Section \ref{sec:polytope} and \ref{sec:complex} respectively, we introduce and study two combinatorial objects that arise naturally in our work: the cocircuit polytope and the external activity complex of a matroid. We study their combinatorial properties, which may be  interesting in their own right, but also play a key role in the proof of our main result, Theorem \ref{longthm cases}. 
We carry out this proof in Section \ref{sec:proofs}. Finally, in Section \ref{sec:affine} we extend our results 
to affine linear spaces. In that case the invariants of $\widetilde{L}$ are controlled by two matroids, and Las Vergnas's Tutte polynomial of a morphism of matroids plays an interesting role.

\subsection{\textsf{Closures of linear spaces.}}\label{sec:closures}

Choose a frame $F = \{\langle e_1\rangle, \ldots, \langle e_n\rangle \}$ where the $e_i$ form a basis of $\kk^n$ and $\langle\  \rangle$ denotes linear span. This allows us to identify $\mathbb{A}^n$ with $\mathbb{A}^1\times \cdots \times \mathbb{A}^1$. The usual embedding of $\mathbb{A}^1$ into $\mathbb{P}^1$ by adding a single point at infinity then gives us an embedding 
$\AA^n \hookrightarrow (\PP^1)^n$. 

\begin{definition}
If $X$ is an affine variety in affine space $\AA^n$, we let $\X$ denote the scheme-theoretic closure $\X$ of $X$ in $(\PP^1)^n$ induced by this embedding $\AA^n \hookrightarrow (\PP^1)^n$. If $I=I(X)$ is the ideal of polynomials vanishing at $X$, we let $\widetilde{I} = I(\X)$ be the ideal of polynomials vanishing at $\X$. 
\end{definition}
%

For the remainder of the paper, we fix a choice of coordinates, and let $S = \kk[x_1,\ldots,x_n]$. The ideals $I(X) \subset S$ 
and 
$I(\widetilde{X})\subset \kk[x_1,\ldots,x_n, y_1,\ldots, y_n]$ of $\X$ are related by
$$I(\widetilde{X}) := (f^h \ | \ f \in I),$$
where $f^h$ is the total homogenization of $f$, obtained by substituting $x_i$ with $\blue{x_i}/y_i$ in $f$ and clearing denominators. 

For general $X$, it does not suffice to only homogenize a set of generators of $I$ to cut out $\widetilde{X}$.  It seems quite difficult to find a canonical presentation of the ideal $I(\widetilde{X})$, or to determine its algebraic invariants, such as the degree, number of generators, or multigraded Betti numbers. However, we show that when $X=L$ is a linear subspace (resp., an affine subspace), all of these questions have elegant answers in terms of the matroid of $L$ (resp., the morphism of matroids), which encodes the relative position of the subspace $L$ with respect to the chosen frame $F$. Let us describe this matroid in two ways.

Our linear space $L$ corresponds to a point in $\mathrm{Gr}(d,n)$, the Grassmannian of $d$-subspaces of $\kk^n$. The choice of a basis $\{e_1, \ldots, e_n\}$ gives an embedding $\pi: \mathrm{Gr}(d,n) \to \mathbb{P}(\wedge^d \kk^n)$ which maps a vector subspace $L$ of $\kk^n$ to its Pl\"ucker vector $\pi(L)$ in $\mathbb{P}(\wedge^d \kk^n)$. Although the coordinates of $\pi(L)$ depend on the choice of basis, the set of coordinate hyperplanes containing $\pi(L)$ only depends on the frame $F$. This set can be identified with the \emph{matroid} $M$ of $L$: for a $d$-subset $S$ of $[n]$, the hyperplane $H_S$ contains $\pi(L)$ if and only if $[n]-S$ is not a basis of $M$. 

More explicitly, if $A$ is an $(n-d) \times n$ matrix whose rows generate the ideal $I=I(L)$ when regarded as linear forms, then the bases of the matroid $M$ are the linearly independent $(n-d)$-subsets of columns of $A$.\footnote{Sometimes the dual choice is made: one may also associate to $L$ the dual matroid of rank $d$, whose bases are the $d$-subsets $S \subset [n]$ such that $H_S$ contains $\pi(L)$. These two choices are equivalent, and we have chosen the one that is more convenient for us.} This matroid will play a key role in what follows.

\subsection{\textsf{Our results on closures of linear spaces.}}\label{sec:results}

%

Given a linear space $L \subset \AA^n$, we are interested in computing various invariants of the closure $\widetilde{L} \subset (\PP^1)^n$ and its ideal $I(\widetilde{L}) \subset \kk[\blue{x_1}, \ldots, \blue{x_n}, y_1, \ldots, y_n]$.  We consider two gradings of $\kk[\blue{x_1}, \ldots, \blue{x_n}, y_1, \ldots, y_n]$ which make the ideal $I(\widetilde{L})$ homogeneous: the \emph{bidegree} with 
\[
\bideg \blue{x_i} = (1,0), \qquad  \bideg y_i = (0,1) \qquad (1 \leq i \leq n) 
\]
and 
the $\ZZ^n$-\emph{multidegree} 
given by
\[
\multideg \blue{x_i}  = \multideg y_i = e_i \qquad (1 \leq i \leq n) 
\]
where $e_i$ is the $i$th unit vector in $\ZZ^n$.

The following theorem shows that the structure of the matroid $M$ of $L$ determines several important geometric, algebraic, and combinatorial invariants of $I(\widetilde{L})$. Conversely, it offers a geometric context where Tutte's basis activities and other matroid invariants appear very naturally. 
We will discuss in detail all the relevant definitions in Section \ref{sec:prelimcombin}.

\begin{theorem}\label{longthm cases} Let $L\subset \AA^n$ be a $d$-dimensional linear space and let $\L \subset (\PP^1)^n$ be the closure of $L$ induced by the embedding $\AA^n \hookrightarrow (\PP^1)^n$. Let $M$ be the matroid of $L$; it has rank $r=n-d$. Then: 
\begin{enumerate}[(a)]
\item The \emph{homogenized cocircuits} of $I(L)$ minimally generate the ideal $I(\widetilde{L})$.
\item The \emph{homogenized cocircuits} of $I(L)$ form a universal Gr\"obner basis for $I(\widetilde{L})$, which is reduced under any term order.  
\item The $\mathbb{Z}^n$-multidegree of $\widetilde{L}$ is
$\sum\limits_{B} t_{b_1}\cdots t_{b_{r}}$ summing over all \emph{bases} $B = \{b_1,\ldots, b_{r}\}$ of $M$.
\item The bidegree of $\widetilde{L}$ is $t^r h_{M}(s/t)$ where $h_M$ is the \emph{$h$-polynomial} of $M$.
\item There are at most $r!\cdot b$ distinct initial ideals of $I(\widetilde{L})$, where $b$ is the number of bases of $M$.
\item 
The initial ideal $\initial_< I(\widetilde{L})$ is the Stanley--Reisner ideal of the external activity complex $B_<(M^*)$ of the dual matroid $M^*$. Its primary decomposition is:
$$\initial_< I(\widetilde{L}) = \bigcap_{B \textrm{ basis }} \left< \, \blue{x_e} \, : \, e  \in IA_<(B)\,  , \, y_e \, : \, e \in IP_<(B) \right>$$
where $B = IA_<(B)\,  \sqcup \, IP_<(B)$ is the partition of $B$ into internally active and passive elements with respect to $<$.
%
\end{enumerate}
\end{theorem}

\begin{remark}\label{rem:order}
A remark is in order about Theorem \ref{longthm cases}(f). An initial ideal $\initial_< I(\widetilde{L})$ is determined by a term order $<$ on $\kk[\blue{x_1}, \ldots, \blue{x_n}, y_1, \ldots, y_n]$. In turn, $<$ leads to a linear order on $[n]$ which we also denote $<$, is defined by $i<j$ for $i, j \in [n]$ whenever $\blue{x_i}y_j > \blue{x_j}y_i$ (or, more revealingly, $\blue{x_i}/y_i > \blue{x_j}/y_j$) in the term order $<$. This is the linear order $<$ with respect to which $IA(B)$ and $IP(B)$ are defined.
\end{remark}


%

We find it remarkable that the matroid $M$, which only contains linear information about $L$, determines so many invariants of the projective variety $\widetilde{L}$. Perhaps this becomes less surprising once we know (a) and (b), which tell us that the form of the defining equations for $\widetilde{L}$ is determined by the matroid. However, in our proof of (a) and (b), we rely on having already computed (using geometric and combinatorial arguments) the invariants of $\widetilde{L}$ and its degenerations in (c) and (f).

\begin{theorem}\label{thm:bettinumbers}
Let $L$ be a linear $d$-space in $\AA^n$, and $I(\widetilde{L})$ the ideal of its closure in $(\mathbb{P}^1)^n$. 
The non-zero multigraded Betti numbers of $S/I(\widetilde{L})$ 
are precisely:
\[
\beta_{i,\a} (S/I(\widetilde{L})) = 
|\mu(F, \widehat{1})|
\]
for each flat $F$ of $M$, where $i=r-r(F)$, and $\a = e_{[n]-F}$.
Here $\mu$ is the M\"obius function of the lattice of flats of $M$.  Furthermore, all of the initial ideals have the same Betti numbers: 
$$\beta_{i,\a} (S/I(\widetilde{L})) = \beta_{i,\a} (S/(\initial_< I (\widetilde{L})))$$
for all $\a$ and for every term order $<$.  
\end{theorem}
As a corollary we obtain the following result.  
\begin{theorem}\label{CM}
If $L$ is a linear space then the ideal $I(\widetilde{L})$ and all of its initial ideals are Cohen-Macaulay. 
\end{theorem}

Before stating the relevant definitions in Section \ref{sec:prelimcombin}, we now briefly introduce them while we discuss an example in detail.

\subsubsection{\textsf{An example.}}

\begin{example}\label{ex} Let $L$ be the subspace of $\AA^6$ cut out by the linear ideal 
\[
I= \left< x_1+x_2+x_6, \,\, x_2-x_3+x_5, \,\, x_3+x_4\right>.
\]
This ideal is given by $r=3$ independent equations in $n=6$ variables, and the corresponding linear subspace $L$ has dimension $d=n-r=3$.

 \end{example}

Consider the $r \times n$ matrix whose rows correspond to our 3 equations:
\[
A
= \begin{bmatrix}
  1 & 1 & 0 &  0 & 0 & 1 \\
  0 & 1 & -1 &  0 & 1 & 0 \\
  0 & 0 & 1 &  1 & 0 & 0
 \end{bmatrix}.
 \]
We regard the columns of $A$ 
as a point configuration in 
$\PP^{r-1} = \PP^2$, respectively, as shown in Figure \ref{fig:points}. 
The affine dependence relations among the points correspond to the linear dependence relations among the columns of the matrix. A different generating set for $I$ would give a different point configuration with the same affine dependence relations.

%

\begin{figure}[ht]
 \begin{center}
  \includegraphics[scale=1]{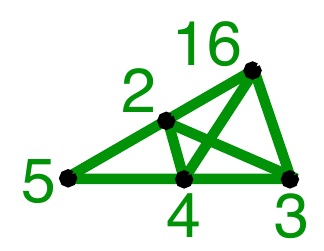}
  \caption{ \label{fig:points} 
A point configuration $A \subset \PP^2$ 
corresponding to the linear ideal $I$.}
 \end{center}
\end{figure}

It is known \cite[Prop. 1.6]{Sturmfels} that the minimal universal Gr\"obner basis of $I$ is given by the \emph{cocircuits} of $I$: the linear forms in $L$ using an inclusion-minimal set of variables.
\[
I=\left< x_1+x_2+x_6, \,\, x_1+x_3-x_5+x_6,\,\, x_1-x_4-x_5+x_6,\,\, x_2-x_3+x_5, \,\,x_2+x_4+x_5, \,\,x_3+x_4\right>
\]
We identify the cocircuits with their support sets 
\[
\D = \{126,\,\, 1356,\,\, 1456,\,\, 235,\,\, 245, \,\,34\} 
\]
They are the complements of the \emph{hyperplanes} $345, 24, 23, 146, 136,$ and $1256$ spanned by subsets of $A$.  Theorem \ref{longthm cases}(a,b) says that the homogenized cocircuits minimally generate $\widetilde{I}$, and give a universal Gr\"obner basis:
\[
\widetilde{I} = \left< \bblue{x_1}y_2y_6+y_1\bblue{x_2}y_6+y_1y_2\bblue{x_6},  \,\, \bblue{x_1}y_3y_5y_6+y_1\bblue{x_3}y_5y_6-y_1y_3\bblue{x_5}y_6+y_1y_3y_5\bblue{x_6}, \ldots,\,\bblue{x_3}y_4+y_3\bblue{x_4}\right>.
\]
%


The \emph{bases} of $A$ are the maximal independent sets of $A$; they correspond to the non-zero maximal minors of $A$, and hence to the 
non-zero Pl\"ucker coordinates of $L$. In Figure \ref{fig:points} they correspond to triples of non-collinear points. 
The $13$ bases of $A$ are 
\[
\B = \{123, 124, 134, 135, 145, 234, 235, 236, 245, 246, 346, 356, 456\}.
\]
Theorem  \ref{longthm cases}(c) then states that the multidegree of $\widetilde{L}$ is 
\[
\textrm{mdeg } \widetilde{L} = t_1t_2t_3 + t_1t_2t_4 + t_1t_3t_4 + 
\cdots
+ t_4t_5t_6.
\]

The \emph{$f$-vector} $f=(1, 6, 14, 13)$ counts the number $f_i$ of independent sets of size $i$ for $i = 1, \ldots, r$. The \emph{$h$-polynomial} is $h_0x^r+h_1x^{r-1}+ \cdots + h_rx^0 = f_0(x-1)^r + f_1(x-1)^{r-1}+\cdots + f_r(x-1)^0$; in this case it equals $x^3+3x^2+5x+4$. Thus Theorem \ref{longthm cases}(d) predicts that 
\[
\bideg \widetilde{L} = 
s^3 + 3s^2t + 5st^2 + 4t^3.
\]

Theorem \ref{longthm cases}(e) says that $I(\widetilde{L})$ has at most $(6-3)! \cdot 13 = 78$ initial ideals. Using the software Gfan \cite{gfan} one can check that it actually has $72$ initial ideals.

Theorem \ref{longthm cases}(f) tells us the primary decomposition of the initial ideal $I_<=in_<I(\widetilde{L})$ with respect to any linear order $<$. If $\bblue{x_1}/y_1 > \cdots > \bblue{x_6}/y_6$, which leads to the natural order $1 < 2 < \cdots < 6$ on the elements of the matroid, we get
%
%
%
\begin{eqnarray*}
I_< &=& \left<\bblue{x_1}y_2y_6, \,\, \bblue{x_1}y_3y_5y_6, \,\,\bblue{x_1}y_4y_5y_6, \,\, \bblue{x_2}y_3y_5, \,\,\bblue{x_2}y_4y_5, \,\,\bblue{x_3}y_4 
\right>
\\\
&=& 
\left<\bblue{x_1},\bblue{x_2},\bblue{x_3}\right> \, \cap \, 
\left<\bblue{x_1},\bblue{x_2},y_4\right> \, \cap \, 
\left<\bblue{x_1},y_3,y_4\right> \, \cap \, 
\left<\bblue{x_1},\bblue{x_3},y_5\right> \, \cap \, 
\left<\bblue{x_1},y_4,y_5\right> \, \cap \, \\
&&
\left<y_2,y_3,y_4\right> \, \cap \, 
\left<y_2,\bblue{x_3},y_5\right> \, \cap \, 
\left<x_2,\bblue{x_3},y_6\right> \, \cap \, 
\left<y_2,y_4,y_5\right> \, \cap \, 
\left<\bblue{x_2},y_4,y_6\right> \, \cap \, \\
&&
\left<y_3,y_4,y_6\right> \, \cap \, 
\left<\bblue{x_3},y_5,y_6\right> \, \cap \, 
\left<y_4,y_5,y_6\right>.
\end{eqnarray*}
We have a primary component $\left<z_b \, : \, b \in B \right>$ for each basis $B$, where $z_b$ equals $\bblue{x_b}$ or $y_b$ depending on whether $b$ is \emph{internally active} or \emph{passive} in $B$. 
For each $b \in B$ consider the cocircuit $D(B,b)$, which consists of the points not on the hyperplane spanned by $B-b$. If $b$ is the smallest element of $D(B,b)$ then $b$ is said to be \emph{active} in $B$, and $z_b=\bblue{x_b}$. Otherwise, $b$ is \emph{passive} in $B$ and $z_b=y_b$.

For example, the basis $235$ contributes the primary component $\left<y_2, \bblue{x_3}, y_5\right>$ because $2$ is internally passive ($2$ is not the smallest element in $D(235,2) = 126$), $3$ is internally active ($3$ is smallest in $D(235,3) = 34$), and $5$ is internally passive ($5$ is not smallest in $D(235,5) = 1456)$.

Note that, from the primary decomposition of $I_<$ above, we can read off the multidegree and bidegree immediately. Each component contributes a monomial, where terms $\bblue{x_i}$ and $y_i$ respectively contribute factors of $s$ and $t$ to the bidegree, and a factor of $t_i$ to the multidegree. Therefore, if one is able to compute this primary decomposition, one immediately gets the list of bases and the $h$-polynomial of the matroid. From this point of view, it is surprising that when we choose different orders $<$ we get the same bidegree.

\begin{figure}[ht]
 \begin{center}
  \includegraphics[scale=.5]{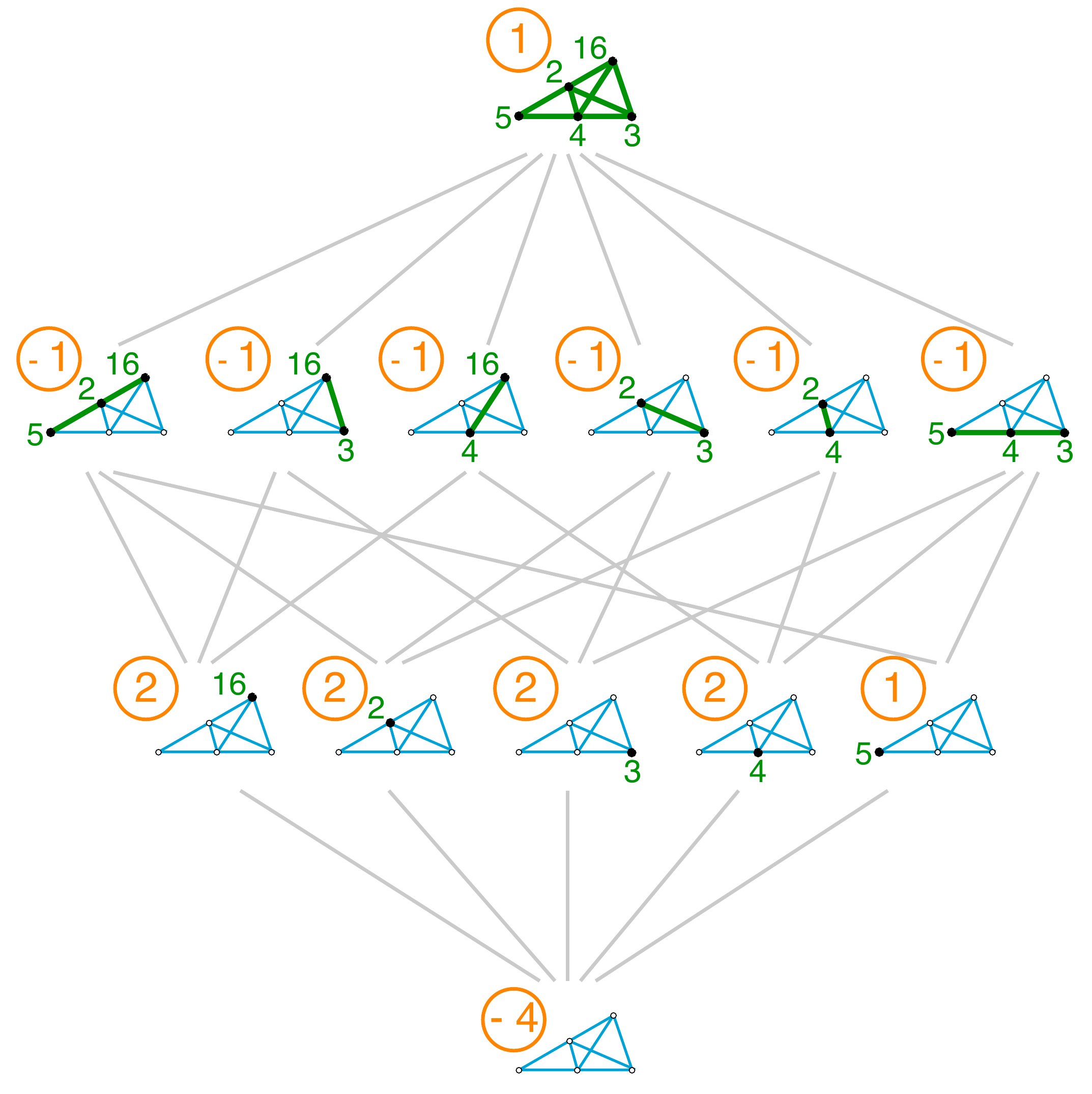}
   \caption{ \label{fig:lattice}
The M\"obius function $\mu(F, \widehat{1})$ of the lattice of flats $M$ encodes the non-zero multigraded Betti numbers of $I(\widetilde{L})$.}
 \end{center}
\end{figure}

\medskip

Theorem \ref{thm:bettinumbers} is best understood pictorially. The \emph{flats} of $M$ are the affine subspaces spanned by the points in $A$. They are partially ordered by inclusion. Recursively define the numbers $\mu(F, \widehat{1})$ by $\mu(\widehat{1}, \widehat{1}) = 1$ and $\mu(F,\widehat{1}) = - \sum_{G > F} \mu(G, \widehat{1})$ for $G \neq \widehat{1}$, where $\widehat{1}$ is the maximal flat. These numbers are shown circled in Figure \ref{fig:lattice}, and they give the non-zero multigraded Betti numbers of $S/I$:
\begin{eqnarray*}
& \beta_{0, \mathbf{\emptyset}} = 1\\
& \beta_{1, \bf{34}} = \beta_{1, \bf{245}} =\beta_{1, \bf{235}} =\beta_{1, \bf{1456}} =\beta_{1, \bf{1356}} =\beta_{1, \bf{126}} = 1 \\
& \beta_{2, \bf{2345}} = \beta_{2, \bf{13456}} = \beta_{2, \bf{12456}} = \beta_{2, \bf{12356}} = 2, \qquad \beta_{2, \bf{12346}} = 1 \\
& \beta_{3,\bf{123456}} = 4
\end{eqnarray*}

\noindent From this we can immediately read off the Betti numbers 
\[
\beta_0 = 1, \quad \beta_1 = 6, \quad \beta_2 = 9, \quad \beta_3 = 4
\]
of $S/I$, as well as the standard $\mathbb{Z}$-graded Betti table of $S/I$, whose $(i,j)$ entry is $\beta_{i,i+j} = \sum\limits_{|\a|=i+j} \beta_{i, \a}$: 
$$\begin{array}{c|cccc}
I & 1 & 6 & 9 & 4\\
\hline 
&1 & - & - & -  \\
& - & 1 & - & -  \\
& - & 3 & 2 & - \\ 
& - & 2 & 7 & 4
 \end{array}$$


In view of Theorem \ref{longthm cases}(b), the equality for $i=1$ in Theorem \ref{thm:bettinumbers} follows from the fact that $I(\widetilde{L})$ is robust; that is, it is minimally generated by a universal Gr\"obner basis. 
For example, the Betti number $\beta_{1, \mathbf{34}}=1$ corresponds to the generator $\bblue{x_3}y_4 + y_3\bblue{x_4}$ of $\widetilde{I}$.



%
%
%

All of these results have generalizations to affine subspaces of $\AA^n$. We delay the precise statements and proofs until Section \ref{sec:affine}.

\subsection{\textsf{Our results on matroids.}}\label{sec:matroidresults}

Our analysis of the closure $\widetilde{L}$ of a linear space $L \subset \AA^n$ in $(\PP^1)^n$ gives rise to some constructions and results in matroid theory  of independent interest.

Fix a basis $e_1, \ldots, e_n$ of $\RR^n$ and let $\Delta = \conv\{e_1, \ldots, e_n\}$ be the standard simplex in $\RR^n$. For each subset $S \subseteq [n]$ consider the indicator vector $e_S  = \sum_{s \in S} e_s$ and the face $\Delta_S = \conv\{e_s \, : \, s \in S\}$ of $\Delta$. 
For a matroid $M$ on $[n]$ consider the \emph{cocircuit polytope}
\[
O_M = \sum_{D \textrm{ cocircuit of }M} \Delta_D,
\]
where the \emph{Minkowski sum} of $P, Q \subset \RR^n$ is $P+Q := \{p + q \, : \, p \in P, q \in Q\}$.

\begin{theorem}\label{thm:polytope}
If a matroid $M$ on $[n]$ has rank $r$, then the cocircuit polytope $O_M$
\begin{enumerate}
\item[(a)]
is given by the equation $\sum_{i=1}^n x_i = D([n])$ and the inequalities $\sum_{i \in S} x_i \leq D(S) \textrm{ for } S \subseteq [n]$, where $D(S)$ is the number of cocircuits intersecting $S$,
\item[(b)]
has dimension $n-c$ where $c$ is the number of connected components of $M$,
\item[(c)]
has the matroid polytope $P_M = \conv\{e_B \, : \, B \textrm{ basis}\}$ as a Minkowski summand,
\item[(d)]
has at most $r! \cdot b$ vertices, where $b$ is the number of bases. 
\end{enumerate}
\end{theorem}

We are also led to the study of an interesting simplicial complex, which we call the \emph{external activity complex}. Let $M$ be a matroid and let $<$ be a linear order on the ground set $S$. Consider the $2|E|$-element set $\{\blue{x_e}, y_e \, : \, e \in E\}$, and identify subsets and monomials, and write
\[
\blue{x_A}y_B := \{\blue{x_a} \, : a \in A\} \cup \{y_b \, : \, b \in B\}.
\]

Every basis $B$ of $M$ gives rise to a partition of $[n]-B$ into the sets $EA_<(B)$ and $EP_<(B)$ of externally active and passive elements. These sets, which will be defined in Section \ref{sec:prelimcombin}, are  similar and related to the sets $IA_<(B)$ and $IP_<(B)$ of Theorem \ref{longthm cases}(e).

\begin{theorem}
Let $M$ be a matroid and $<$ a linear order on its ground set. There is a simplicial complex $B_<(M)$ on $\{\blue{x_e}, y_e \, : \, e \in E\}$, called the \emph{external activity complex}, such that
\begin{enumerate}
\item
The minimal non-faces are $\blue{x_{\min C}}  y_{C-\min C}$ for each circuit $C$.
\item
The facets of $B_<(M)$ are the sets $\blue{x_{B \cup EP(B)}}  y_{B \cup EA(B)}$ for each basis $B$.
\end{enumerate}
\end{theorem}

%
%

\subsection{\textsf{Motivation and related results.}}\label{sec:related}

The closures we study are similar to the reciprocal varieties of linear spaces considered in \cite{PS}. A reciprocal variety may be thought of as the closure arising from the homogenization $x_i \mapsto 1/y_i$. Proudfoot and Speyer proved that the reciprocal variety $L^\perp$ of a linear space has a universal Gr\"obner basis defined by circuit polynomials.  They also proved that its degree can be computed in terms of the Tutte polynomial evaluated at $(1,0)$.  In one sense, our results can be viewed as a proper homogenization of the reciprocal variety.  Indeed, upon setting $\blue{x_i} = 1$ we obtain precisely the equations obtained in \cite{PS}.  Upon substitution the minimality of the generators is not preserved, nor is the property that all monomial degenerations have the same Betti numbers.  The bidegree of $I(\widetilde{L})$ is a homogenised $h$-polynomial whereas the degree of the reciprocal variety is equal to the constant term $h(0)$.  Thus it seems that the added homogeneity enjoyed by the closure in $(\mathbb{P}^1)^n$ captures more of the matroid structure.

We originally became interested in closures of linear spaces because of this universal Gr\"obner basis property and a well-known result for toric ideals:  If $X$ is any affine toric variety, then its closure $\widetilde{X}$ in $(\mathbb{P}^1)^n$ is called the Lawrence lifting of $X$.  If $X$ is toric then $I(\widetilde{X})$ is minimally generated by a universal Gr\"obner basis (see \cite{Sturmfels}).  In general, we wanted an answer to the following 

\begin{question}If $X\subset \mathbb{A}^n$ is a variety, and $I(\widetilde{X})$ is the ideal of its closure in $(\mathbb{P}^1)^n$, when does 
$$\beta_i (S/I(\widetilde{X})) = \beta_i (S/\initial_< I(\widetilde{X})) \mbox{ for all $<$ }?$$
\end{question}
Sturmfels' result says that equality holds if $X$ is defined by a toric ideal and $i = 1$, and our result says that if $L$ is a linear space then equality holds for all $i$.

Originally we hoped that such a result might hold more generally, but little is true.  Even for toric ideals, the question has a negative answer if $i\geq 2$ and the following example 
shows that even for $i=1$ the situation is quite subtle.   
This example also illustrates that, contrary to the case of closures in $\mathbb{P}^n$, there is no simple numerical relationship between the number of generators of $I(X)$ and $I(\widetilde{X})$, even in terms of Gr\"obner bases.

\begin{example}\label{counterex}  Let $I = I(X) = (x_1 + x_2 + x_3, \,\, x_1 + x_3 + x_4, \,\, x_1^2 + x_2^2+ x_1x_4)$
$$\begin{array}{|r|ll|}
\hline
&  I(X)\,\,\, &  I(\widetilde{X}) \\
\hline
\mbox{number of generators}  &  3 & 12 \\
\mbox{size of a reduced Gr\"obner basis}  & 3 & 14 \mbox{ or } 15\\
\mbox{size of a universal Gr\"obner basis}& 8 & 21\\
\hline
\end{array}$$
\end{example}
The closures we study also arise in work of Aholt, Sturmfels, and Thomas \cite{AST} where they study maps induced by products of linear projections $V \to V/V_i$ when $\dim V = 4$.   Recently, Li has extended the results of \cite{AST} to arbitrary vector spaces.  In \cite{Li} he computes the defining ideal and multi-degree for the closure of the image of such maps and proves that they are determined combinatorially.

Ideals minimally generated by universal Gr\"obner bases, called \emph{robust ideals} in \cite{BR}, are by no means a common occurrence.  Even in the toric case, this condition is very strong, yet a complete classification is unknown.  Nonetheless, robust ideals have cropped up in many classical situations; see \cite{B, BR, CGD, CHT,PS, SZ}.


\section{\textsf{Preliminaries from matroid theory.}}\label{sec:prelimcombin}

The toolkit of matroid theory is ideally suited to study the geometric and algebraic invariants in this project. Matroid theory can be approached from many equivalent points of view. This can make the theory confusing at first, as different papers often use very different definitions of a matroid. However, in the long run, the existence of these ``cryptomorphic" definitions is an extremely powerful feature of the theory.
This project illustrates this point very well; many different matroid theoretic concepts appear naturally, as Example \ref{ex} shows. In this section we introduce these concepts in more detail; they will play a fundamental role in what follows. For a more thorough introduction, see \cite{Bj, Oxley}.


\subsection{\textsf{One definition of a matroid.}}

\begin{definition}
A \emph{matroid} $M=(E, \mathcal{I})$ consists of a ground set $E$ and a family $\mathcal{I}$ of sets of $E$, called the \emph{independent sets} of $M$, which satisfy the following axioms:

\noindent
(I1) The empty set is independent.

\noindent
(I2) A subset of an independent set is independent.

\noindent
(I3) If $X$ and $Y$ are independent and $|X|<|Y|$, then there exists $y \in Y-X$ such that $X \cup y$ is independent.
\end{definition}

Matroid theory can be thought of as a combinatorial theory of independence. The prototypical example is the family of \emph{linear} or \emph{realizable matroids}, which arise from linear independence. If $E$ is a set of vectors in a vector space $V$, then the linearly independent subsets of $E$ form a matroid. 

In a matroid $M$, a \emph{circuit} is a minimal dependent set. 
A \emph{basis} is a maximal independent set. All bases of $M$ have the same size, which is called the \emph{rank} $r(M)$ of $M$. Similarly, all maximal independent subsets of any set $S\subseteq E$ have the same size, which is called the rank $r(S)$. There are equivalent definitions of matroids in terms of circuits, bases, and rank functions, among others.

In our running Example \ref{ex}, the bases and circuits are:
\begin{eqnarray*}
\B &=& \{123, 124, 134, 135, 145, 234, 235, 236, 245, 246, 346, 356, 456\},\\
\C &=& \{16, 125, 256, 345, 1234, 2346\}.
\end{eqnarray*}
The independent sets are the subsets of the bases.

\subsection{\textsf{The $f$-vector, $h$-vector, and Tutte polynomial.}}

The \emph{$f$-vector} $f_M=(f_0, \ldots, f_r)$ of a matroid $M$ records the number $f_i$ of independent sets of size $i$ for $0 \leq i \leq r$. This information is equivalently recorded in the \emph{$h$-polynomial}
\[
h_M(x) = h_0x^r+h_1x^{r-1}+ \cdots + h_rx^0 = f_0(x-1)^r + f_1(x-1)^{r-1}+\cdots + f_r(x-1)^0.
\]
The reverse polynomial $h_rx^r+\cdots + h_0$ is known as the \emph{shelling polynomial} of $M$. The vector $h_M = (h_0, \ldots, h_r)$ is called the \emph{$h$-vector} of $M$. 

In our running Example \ref{ex} we already saw that there are 13 bases. All sets of size 0, 1, and 2 are independent except for the pair 16, so the $f$-vector is $(1, 6, 14, 13)$. The $h$-polynomial is then
\[
h_M(x) = (x-1)^3 + 6(x-1)^2+14(x-1)+13 = x^3 + 3x^2 + 5x + 4.
\]

The $h$-polynomial is an evaluation of the most important enumerative invariant of a matroid, the \emph{Tutte polynomial}:
\[
T_M(x,y) = \sum_{A \subseteq E} (x-1)^{r-r(A)} (y-1)^{|A|-r(A)}.
\]
A straightforward computation shows that 
\[
h_M(x) = T_M(x,1).
\]

\subsection{\textsf{Duality and minors.}}
If $\B$ is the collection of bases of a matroid $M$ on $E$, then $\B^*:=\{E-B \, : \, B \in \B\}$ is also the collection of bases of a matroid, called the \emph{dual matroid} $M^*$. If $M$ is the matroid of a set $A$ of $n$ vectors which generate $\kk^d$, then one can find a set of $n$ vectors which generate $\kk^{n-d}$ whose matroid is $M^*$. Figure \ref{fig:dual} shows a point configuration dual to the one of Example \ref{ex}. The reader may check that the bases of $A^*$ are precisely the complements of the bases of $A$.

A \emph{circuit} of $M^*$ is called a \emph{cocircuit} of $M$. It can also be characterized as a minimal set $D$ whose removal decreases the rank of $M$; i.e., $r(E-D)<r$. The cocircuits of $A$ in Example \ref{ex} are
\[
\D = \{34, 126, 235, 245, 1356, 1456\}.
\]
They are the complements of the hyperplanes spanned by subsets of $A$. 
In $A^*$ they are the circuits. 




\begin{figure}[ht]
 \begin{center}
  \includegraphics[scale=1]{dualvectorconfig}
  \qquad \qquad \qquad 
  \includegraphics[scale=1]{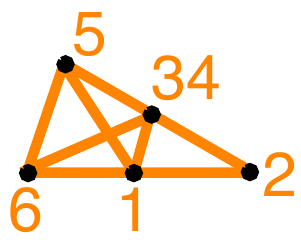} 
  \caption{ \label{fig:dual} 
The point configuration $A \subset \PP^{r-1} = \PP^2$ 
and a dual configuration $A^* \subset \PP^{n-r-1} = \PP^2$.}
 \end{center}
\end{figure}

%
%


The following technical lemma will be very useful to us.

\begin{lemma}\cite{Oxley}\label{lemma:circuitcocircuit}
If $C$ is a circuit and $D$ is a cocircuit of $M$, then $|C \cap D| \neq 1$.
\end{lemma}

If $M$ is a matroid on $E$ and $A \subset E$ then there are matroids $M\backslash A = M|_{E-A}$ and $M/A$ on $E-A$, called the \emph{deletion} and \emph{contraction} of $A$ in $M$,  whose independent sets are
\begin{eqnarray*}
\II(M\backslash A) &=& \{I \in \II(M) \, : \, I \subseteq E-A\}\\
\II(M/A) &=& \{I-B_A \, : I \in \II(M), B_A \subseteq I\}
\end{eqnarray*}
where $B_A$ is a basis of $A$. Any sequence of deletions and contractions commutes. A \emph{minor} of $M$ is any matroid obtained from $M$ by deletions and contractions.

Deletion and contraction are dual operations:
\[
(M/A)^* = M^* \backslash A.
\]
If $M$ comes from a set $S$ of vectors in a vector space $V$, then $M \backslash A$ corresponds to deleting the vectors in $A$, while $M/A$ corresponds to the images of those vectors in $V/\textrm{span}(A)$.

\subsection{\textsf{The matroid of a linear ideal.}}

Fix a choice of a standard basis for $\kk^n$. Let $L$ be an $r$-dimensional linear subspace of $\kk^n$ and let $I(L) \subset \kk[x_1,\ldots, x_n]$ be its defining linear ideal. 
There is one particularly useful generating set for $I(L)$, which we now describe.  For each linear form $f$ in $I(L)$ consider its \emph{support} $\supp(f) \subseteq [n]$ consisting of those $i$ such that $x_i$ has a nonzero coefficient in $f$. Among these, consider the set $\D$ of inclusion-minimal supports; these are called the \emph{cocircuits} of $I(L)$. They are the cocircuits of a matroid $M(L)$, called the \emph{matroid of $L$}.\footnote{Sometimes the dual convention is chosen, and the matroid of $L$ is defined to be the dual matroid $M(L)^*$.}
Notice that for each cocircuit $D$ there is a unique linear form $f$ (up to scalar multiplication) in $I(L)$ with $\supp(f)=D$, so there is no ambiguity in calling this form $f$ a cocircuit as well.  

\begin{proposition}\label{prop:Grobnerlinear}\cite[Prop. 1.6]{Sturmfels}
The cocircuits of the linear ideal $I(L)$ form a universal Gr\"obner basis for $I(L)$.
\end{proposition}

Linear matroids are precisely the matroids of linear ideals. As we explained in Example \ref{ex}, if $B$ is a matrix whose rows generate $I(L)$, one may easily check that the linear matroid on the columns of $B$ equals the matroid of $L$.

Matroid duality can then be seen as a generalization of duality of subspaces. Our chosen basis for $\kk^n$ determines a dual basis for the dual vector space $(\kk^n)^*$.
If $L^\perp \subset (\kk^n)^*$ is the orthogonal complement of our vector space $L$, then the matroid of $L^\perp$ is dual to the matroid of $L$.

\subsection{\textsf{Basis activities.}}
\begin{proposition}\cite{Crapo}
Given a basis $B$ and an element $x \notin B$, there is a unique circuit $C=C(B,x)$ contained in $B \cup x$. It is called the \emph{fundamental circuit} of $B$ and $x$, and is given by:
\[
C(B,x) = \{y \in E \, : \, (B \cup x) -  y \textrm{ is a basis}\}
\]
Notice that $x \in C(B,x)$.

Given a basis $B$ and an element $y \in B$, there is a unique cocircuit $D=D(B,y)$ contained in $E-B \cup y$. It is called the \emph{fundamental cocircuit} of $B$ and $y$, and is given by:
\[
D(B,y) = \{x \in E \, : \, (B-y) \cup x \textrm{ is a basis}\}
\]
Notice that $y \in D(B,y)$.
\end{proposition}

Consider a linear order $<$ on the ground set of $M$. Let $B$ be a basis of $M$. We say that an element $e \notin B$ is \emph{externally active} if it is the smallest element in $C(B,e)$, and it is \emph{externally passive} otherwise. Let $EA_<(B)$ and $EP_<(B)$ be the sets of externally active and externally passive elements with respect to $B$ and $<$.\footnote{When the choice of the order $<$ is clear, we will omit the subscript and write simply $EA(B)$ and $EP(B)$.}

Similarly, we say that an element $i \in B$ is \emph{internally active} if it is the smallest element in $D(B,i)$, and it is \emph{internally passive} otherwise. We write $IA_<(B)$ and $IP_<(B)$ for the sets of internally active and internally passive elements with respect to $B$ and $<$. Notice that matroid duality reverses internal and external activity: $IA_{M,<}(B) = EA_{M^*,<}([n]-B)$ and $IP_{M,<}(B) = EP_{M^*,<}([n]-B)$. 

We will need the following result by Crapo:

\begin{theorem}\cite{Crapo}\label{thm:Crapo}
Let $M$ be a matroid on the ground set $S$ and let $<$ be a linear order on $S$. Every subset $A$ of $S$ can be uniquely written in the form $A=B \cup E - I$ for some basis $B$, some subset $E \subseteq EA(B)$, and some subset $I \subseteq IA(B)$. 
Equivalently, the intervals $[B - IA(B), B \cup EA(B)]$ form a partition of the poset $2^S$ of subsets of $S$ ordered by inclusion.
\end{theorem}

This can be used to prove:

\begin{theorem} \cite{Crapo} \label{thm:Crapo}
Let $M$ be a matroid on the ground set $S$ and let $<$ be a linear order on $S$. Then the Tutte polynomial and $h$-polynomial of $M$ are given by
\[
T_M(x,y) = \sum_{B \, \textrm{basis}} x^{|IA_<(B)|} y^{|EA_<(B)|}, \qquad 
h_M(x) = \sum_{B \, \textrm{basis}} x^{|IA_<(B)|}. 
\]
\end{theorem}

This beautiful result implies, in particular, the nontrivial fact that the right hand side of each equation does not depend on the chosen linear order. 

\subsection{\textsf{Lattice of flats and M\"obius function.}}

A \emph{flat} $F$ of a matroid $M$ is a subset which is maximal for its rank; that is, a set such that $r(F \cup f) = r(F)+1$ for all $f \notin F$. The flats of rank $r-1$ are called \emph{hyperplanes}. In the case that interests us, when $M$ is the linear matroid of a set of vectors $E \subset \kk^n$, the flats of $M$ correspond to the subspaces spanned by subsets of $E$. 
The \emph{lattice of flats} $L_M$ is the poset of flats ordered by containment; it is in fact a lattice, graded by rank. The flats in Example \ref{ex} are
\[
L_M = \{\emptyset, 16, 2, 3, 4, 5, 1256, 136, 146, 23, 24, 345, 123456\},
\]
as illustrated in Figure \ref{fig:lattice}.


The \emph{M\"obius function} of $L_M$ is the map $\mu: Int(L_M) \rightarrow \ZZ$ from the intervals of $L_M$ to $\ZZ$ characterized by $\mu(x,x) = 1$ for all $x \in L_M$ and $\sum_{x \leq z \leq y} \mu(z,y) = 0$ for all $x < y$.\footnote{It is more common to demand that $\sum_{x \leq z \leq y} \mu(x,z) = 0$ for all $x<y$; these two conditions are equivalent.} The \emph{M\"obius number} of $M$ is $\mu(M) = \mu(\widehat{0}, \widehat{1})$, where $\widehat{0}$ and $\widehat{1}$ are the minimum and maximum elements of $L_M$. Figure \ref{fig:lattice} shows the value of $\mu(F,\widehat{1})$ next to each flat $F$ of $M$.

\subsection{\textsf{Independence complexes and cocircuit ideals.}}\label{sec:complexes}

To a matroid $M$ on the ground set $E$ one associates a simplicial complex
\[
IN(M) = \{I \subseteq E \, : \, I \textrm{ is independent in } M\}
\] 
called the \emph{independence complex} of $M$. For us, the independence complex of the dual matroid $M^*$ is more relevant. These complexes have very simple topology:

\begin{theorem} \label{thm:Bjorner}
 \cite[Theorem 7.8.1]{Bj} If $M$ is a matroid of rank $r$ on $[n]$, then 
 \[
H_i(IN(M^*); \ZZ) = \begin{cases} \ZZ^{\, |{\mu}(M)|}, & \mbox{if } i=n-r-1 \mbox{ and $M$ has no loops}\\ 0 , & \mbox{otherwise}. \end{cases}.
\]
\end{theorem}

Recall that the \emph{Stanley-Reisner ideal} of a simplicial complex $\Delta$ on a set $\{x_1,  \ldots, x_n\}$ is the ideal 
\[
I_\Delta = \left<x_{i_1} x_{i_2} \cdots x_{i_k} \, : \, \{i_1, \ldots, i_k\} \textrm{ is not a face of } \Delta\right> \subset \kk[x_1, \ldots, x_n].
\]
The \emph{Stanley-Reisner ring} is $\kk[x_1, \ldots, x_n]/I_\Delta$.
Since the minimal non-faces of $IN(M^*)$ are the circuits of $M^*$, which are the cocircuits of $M$, the Stanley--Reisner ideal of $IN(M^*)$ is the \emph{cocircuit ideal}
\[
I_{IN(M^*)} = \left< \prod_{c \in C} x_c \, : C \textrm{ is a cocircuit of } M\right>.
\]
The components of the primary decomposition of a squarefree monomial ideal $I_\Delta$ are in bijection with the facets of $\Delta$; each facet $F$ corresponds to the primary component $\left<x_f \, : \, f \notin F\right>$.  \cite[Theorem 1.7]{MS} Since the facets of $IN(M^*)$ are the bases of $M^*$, we get that the primary decomposition of $I_{IN(M^*)}$ is
\[
I_{IN(M^*)} = \bigcap_{B \textrm{ basis}} \left<x_b \, : \,  b \in B\right>.
\]

In our running Example \ref{ex} we have 
\begin{eqnarray*}
I_{IN(M^*)} &=& \left< x_1x_2x_6, \,\, x_2x_3x_5, \,\,x_2x_4x_5, \,\,x_3x_4, \,\,x_1x_3x_5x_6, \,\,x_1x_4x_5x_6 \right> \\
&=& 
\left<x_1,x_2,x_3\right> \, \cap \, 
\left<x_1,x_2,x_4\right> \, \cap \, 
\left<x_1,x_3,x_4\right> \, \cap \, 
\left<x_1,x_3,x_5\right> \, \cap \, 
\left<x_1,x_4,x_5\right> \, \cap \, \\
&&
\left<x_2,x_3,x_4\right> \, \cap \, 
\left<x_2,x_3,x_5\right> \, \cap \, 
\left<x_2,x_3,x_6\right> \, \cap \, 
\left<x_2,x_4,x_5\right> \, \cap \, 
\left<x_2,x_4,x_6\right> \, \cap \, \\
&&
\left<x_3,x_4,x_6\right> \, \cap \, 
\left<x_3,x_5,x_6\right> \, \cap \, 
\left<x_4,x_5,x_6\right>.
\end{eqnarray*}

Now we recall Hochster's formula for the Betti numbers of a squarefree monomial ideal:

\begin{theorem}\cite[Corollary 5.12]{MS}
The nonzero Betti numbers of the Stanley--Reisner ring $I_\Delta$ lie only in squarefree degrees $\sigma$, and
\[
\beta_{i-1, \sigma}(I_\Delta) = \dim_\k \widetilde{H}^{|\sigma|-i-1}(\Delta |_\sigma).
\]
\end{theorem}
Let us apply this formula to $\Delta = IN(M^*)$ and $\sigma = E-A$ for a subset $A \subset E$. We have that $\Delta |_{E-A} = IN(M^*|_{E-A}) = IN(M^* \backslash A)$. 
Notice that $(M^* \backslash A)^* = M/A$ has no loops if and only if $A$ is a flat of $M$. Also $r(M/A) = r - r(A)$ and $\mu(M/A) = \mu(A, \widehat{1})$. Combining these observations with Theorem \ref{thm:Bjorner}, we obtain the following result.

\begin{theorem}\label{thm:onlynonzeroBetti}
The only nonzero Betti numbers of the cocircuit ideal $I_{IN(M^*)}$ of $M$ are
\[
\beta_{r-r(A)-1, e_{E-A}}(I_{IN(M^*)}) = |\mu(A, \widehat{1})| 
\]
for the flats $A$ of $M$.
\end{theorem}

%

\section{\textsf{Preliminaries from commutative algebra.}}\label{sec:prelimalg}

In this section we briefly review relevant definitions and results from combinatorial commutative algebra.  We refer the reader to \cite{HerzogHibi, MS} for a thorough treatment of all of these topics.

\subsection{\textsf{Free resolutions and Betti numbers.}}\label{sec:freeres}
Recall that $S = \kk[x_1,\ldots,x_n]$. Let $M$ be a graded $S$-module. One important invariant of $M$ is its \emph{minimal free resolution}, which is an exact sequence of maps of $S$-modules
$$\xymatrix{0 \ar[r] &  F_d \ar[r]^{\phi_{d}} & F_{d-1} \ar[r]^{\phi_{d-1}} &  \cdots \ar[r]^{\phi_{1}} &  F_0 \ar[r]^{\phi_{0}} &  M \ar[r] & 0}$$
where the $F_i$ are free modules chosen to have rank as small as possible.  Such a resolution is unique up to isomorphism, and the maps can be chosen so that they are homogeneous with respect to the grading.  Each $F_i$ is a direct sum
$$F_i = \bigoplus_{\mathbf{a}} S(-\mathbf{a})^{\beta_{i,\mathbf{a}}}$$
where $S(-\mathbf{a})$ denotes the free module whose generator lies in degree $\mathbf{a}$.  The ranks of these graded pieces are defined to be the \emph{graded Betti numbers} $\beta_{i,\mathbf{a}}$ and can be computed as dimensions of $\Tor$ modules according to the formula
$$\beta_{i,\mathbf a}(M) =  \dim_k (\Tor^S_i(M,\kk))_{\mathbf a}.$$
The modules $F_i$ in the minimal free resolution are collectively called \emph{syzygies} and encode algebraic relations among the generators of $M$. In what follows we assume that $M = S/I$ for a homogeneous ideal $I$.

\begin{example}
With this notation, the ideal $\widetilde{I}$ in Example \ref{ex} has minimal free resolution 
$$\xymatrix{0 \ar[r] &  
{\begin{array}{c}
S(-\mathbf{123456})^4
\end{array}}
\ar[r] &
{\begin{array}{c}
S(-\mathbf{12346}) \\  \oplus \\
S(-\mathbf{2345})^2 \\  \oplus \\
S(-\mathbf{13456})^2 \\ \oplus \\
S(-\mathbf{12456})^2 \\  \oplus \\
S(-\mathbf{12356})^2
\end{array}}
 \ar[r] &  
{\begin{array}{c}
S(-\mathbf{126}) \\ \oplus \\
S(-\mathbf{34}) \\  \oplus \\
S(-\mathbf{1356}) \\  \oplus \\
S(-\mathbf{1456}) \\  \oplus \\
S(-\mathbf{235}) \\  \oplus \\
S(-\mathbf{245})
\end{array}}
\ar[r] & S \ar[r] &  S/\widetilde{I} \ar[r] & 0},$$
where we denote, for example, the degree $(1,1,0,0,0,1)$ as $\mathbf{126}$ to save space. We have $\beta_{1,\mathbf{34}} = 1$ because $\widetilde{I}$ has one generator in degree $\mathbf{34}$, namely, the homogenized cocircuit $\blue{x_3}y_4+y_3\blue{x_4}$. We have $\beta_{2,\mathbf{2345}} = 2$ because there are two independent $S$-linear relations (\emph{syzygies}) in degree $\mathbf{2345}$ among the six generators of $\widetilde{I}$, namely:
\begin{eqnarray*}
y_2y_5(\bblue{x_3}y_4+y_3\bblue{x_4})+
y_4(\bblue{x_2}y_3y_5 - y_2\bblue{x_3}y_5+y_2y_3\bblue{x_5})
-y_3(\bblue{x_2}y_4y_5 + y_2\bblue{x_4}y_5+y_2y_4\bblue{x_5}) &=& 0, \\
(\bblue{x_2}y_5+y_2\bblue{x_5})(\bblue{x_3}y_4+y_3\bblue{x_4})-
\bblue{x_4}(\bblue{x_2}y_3y_5 - y_2\bblue{x_3}y_5+y_2y_3\bblue{x_5})
-\bblue{x_3}(\bblue{x_2}y_4y_5 + y_2\bblue{x_4}y_5+y_2y_4\bblue{x_5}) &=& 0.
\end{eqnarray*}

\end{example}

A \emph{monomial term order} $<$ on the polynomial ring $S$ is a total order on the set of monomials  that is respected by multiplication.  For each polynomial $f$, the term order determines a \emph{leading term} $\initial_<f$ which is the largest term with respect to $<$.  For an ideal $I\subset S$ we define the \emph{initial ideal} with respect to $<$ to be the ideal generated by the leading terms of all polynomials in $I$:
$$\initial_< I = \left( \initial_<f \ | \ f\in I \right).$$
Initial ideals of can be thought of as flat degenerations (see \cite[Theorem 15.17]{E}) and as such, the Hilbert function of an ideal is equal to that of its initial ideal.  By contrast, Betti numbers may change, but they obey the following inequality:
$$\beta_{i, \mathbf{a}}(S/I) \leq \beta_{i, \mathbf{a}}(S/\initial_< I).$$
The upshot, however, is that if this inequality is strict, then the extra free modules appearing in the minimal free resolution of $S/\initial_< I$ must occur in pairs; the modules in each pair occur in neighboring homological degrees and have the same generating degree.  Such a pair is known as a \emph{consecutive cancellation}, because in the resolution of $S/I$, these modules cancel out. (See \cite{Peeva:Book} and \cite[Remark 8.30]{MS})

\begin{example}
If $J = \left<x^2,xy+y^2\right>\subset \kk[x,y]$ then under the term order determined by $x>y$,  
$$\initial_<(J) = \left<x^2,xy,y^3\right>.$$
Under the usual $\mathbb{Z}$-grading, the ideals $J$ and $\initial_<J$ have minimal free resolutions

$$\xymatrix{0 \ar[r] &
S(-4)^1
 \ar[rr]^{\left(\begin{array}{c}
 xy + y^2 \\ -x^2
 \end{array}\right)}
 &  &
S(-2)^2
\ar[rrr]^{
\left(\begin{array}{cc}
x^2 & xy + y^2 
\end{array}\right)
}
 & & & S \ar[r] & S/J \ar[r] & 0}.$$
and 
$$\xymatrix{0 
\ar[r] &
{\begin{array}{c}
S(-3)^1\\ \oplus \\ S(-4)^1 \end{array}}
 \ar[rr]^{\left(\begin{array}{cc}
 y &0 \\
 -x & y^2\\ 
  0 & -x 
 \end{array}\right)}
 &  &
{\begin{array}{c}
S(-2)^2\\ \oplus \\ S(-3)^1 \end{array}}
\ar[rrr]^{
\left(\begin{array}{ccc}
x^2 & xy & y^3
\end{array}\right)
}
 & & & S \ar[r] & S/\initial_< J  \ar[r] & 0}.$$ 
The consecutive pair of two copies of $S(-3)^1$ is a consecutive cancellation.
\end{example}

\subsection{\textsf{Degree and multidegree.}}\label{sec:deg}

If $X \subset \mathbb{P}^n$ is a projective variety over an algebraically closed field then the \emph{degree} of $X$ is defined to be the number of intersection points of $X$ with a linear subspace $L$ in general position where $\dim X + \dim L = n$.  Since we work in a product of projective spaces, we will consider a finer invariant, called the \emph{multi-degree}.  In a product of projective spaces, the multi-degree captures the number of points of intersection of $X$ with general collections of linear spaces in the different coordinate subspaces.  It is convenient to encode these numbers as the coefficients of a polynomial.  We define multidegree geometrically for varieties inside of $(\mathbb{P}^1)^n$ and refer the reader to \cite{MS} for the more general case, as well as an algebraic definition in terms of free resolutions. 

Let $d = \dim X$ and $r=n-d = \codim X$. For each $r$-subset $\Delta\subset [n]$, consider the linear subspace $Z_\Delta\subset (\PP^1)^n = \prod_{i = 1}^n (\PP^1)_i$  (where we give subindices to the various $\PP^1$s to distinguish them) given by
\[
Z_\Delta = \prod_{i\in \Delta} (\mathbb{P}^1)_i \times \prod_{i\notin \Delta} q_i,
\]
where $q_i\in (\PP^1)_i$ is a general point.  If $X$ is a subvariety of $(\mathbb{P}^1)^n$ of codimension $r$ then denote by $m(Z_\Delta, X)$ the intersection multiplicity of $X$ with $Z_\Delta$.   By the genericity of $Z_\Delta$ this will simply be the number of points in the intersection counted with multiplicity.  Then the \emph{multi-degree} of $X$ is defined to be the polynomial 
$$\textrm{mdeg } X = \sum_{\Delta \in {[n] \choose r} }m(Z_\Delta, X) \, t_{i_1}\cdots t_{i_r},$$
where $\Delta=\{i_1, \ldots, i_r\}$ ranges over all $r$-subsets of $[n]$.

By definition, the multi-degree can only detect information about the highest dimensional components of $X$.  The next result follows easily from the definitions.
\begin{proposition}\label{prop:mdeg}
Let $I\subset \kk[\blue{x_1},\ldots, \blue{x_n}, y_1, \ldots, y_n]$ be an ideal defining a subscheme $X$ in $(\mathbb{P}^1)^n$.  If $X$ has irreducible components of maximal dimension $\{X_1, \ldots, X_k\}$ then 
$$\textrm{mdeg } X  = \sum_{i=1}^k \textrm{mdeg }X_i.$$
If $X$ is defined by a monomial ideal $I = (z_1^{a_1},\ldots, z_c^{a_c})$ where $z_i$ is either $\blue{x_i}$ or $y_i$ and $a_i \in \mathbb{N}$ for each $1 \leq i \leq c$, then 
$$\textrm{mdeg } X = (a_1\cdots a_c) \,\, t_1\cdots t_c.$$
\end{proposition}

\section{\textsf{The cocircuit polytope.}}\label{sec:polytope}


Our work gives rise to an interesting polytope $O_M$ associated to a matroid $M$, which we call the \emph{cocircuit polytope}. As we will see in the proof of Theorem \ref{longthm cases}(e), when $M$ is the matroid of a linear space $L$ the polytope $O_M$ is affinely isomorphic to the state polytope of the ideal $I(\widetilde{L})$.

Let $\Delta = \conv\{e_1, \ldots, e_n\}$ be the standard simplex in $\RR^n$, and for each $I \subseteq [n]$ let 
\[
\Delta_I = \conv\{e_i \, : \, i \in I\}.
\]
For a matroid $M$ on the ground set $[n]$, let the \emph{cocircuit polytope} $O_M$ be the Minkowski sum
\[
O_M = \sum_{D \textrm{ cocircuit of }M} \Delta_D,
\]
where the \emph{Minkowski sum} of $P, Q \subset \RR^n$ is $P+Q := \{p + q \, : \, p \in P, \, q \in Q\}$.

We will see that these polytopes $O_M$ are related to \emph{matroid (basis) polytopes}, which are much better known and understood; see for example \cite{Edmonds, GGMS}. The vertices of the \emph{matroid polytope} $P_M$ of $M$ are the vectors $e_B = e_{b_1} + \cdots + e_{b_r}$ for each basis $B=\{b_1, \ldots, b_r\}$ of $M$. 
The \emph{connected components} of $M$ are the equivalence classes for the equivalence relation where $a \sim b$ if $a, b \in C$ for some circuit $C$. It is known that $\dim P_M = n - c$ where $c$ is the number of connected components of $M$.

Figure \ref{fig:polytopes} illustrates these polytopes for the matroid $M$ with bases $12$, $13$, $14$, $23$, and $24$. The top left panel shows the standard simplex as a frame of reference. The figure builds up the Minkowski sum $O_M = \Delta_{134} + \Delta_{234} + \Delta_{12}$ one step at a time. It then subtracts the simplex $\Delta_{1234}$ to obtain the signed Minkowski sum $P_M = \Delta_{134} + \Delta_{234} + \Delta_{12} - \Delta_{1234}$, which is  the matroid polytope $P_M$ as shown in \cite{ABD}. 

\begin{figure}[ht]
 \begin{center}
  \includegraphics[scale=0.66]{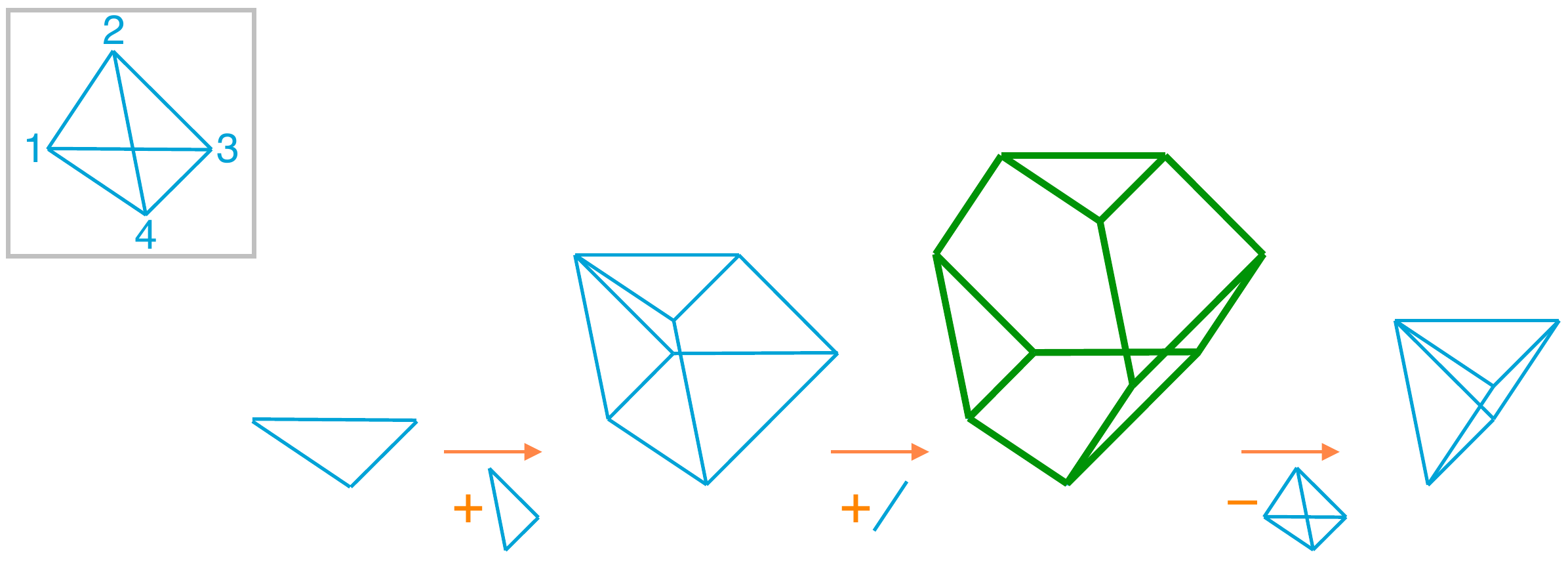}
  \label{fig:polytopes}
  \caption{
  Building up the Minkowski sum
 $\Delta_{134} + \Delta_{234} + \Delta_{12} = O_M$
and the signed Minkowski sum  $\Delta_{134} + \Delta_{234} + \Delta_{12} - \Delta_{1234} = P_M$ one step at a time.}
 \end{center}
\end{figure}

\begin{reptheorem}{thm:polytope}
If a matroid $M$ on $[n]$ has rank $r$, then the cocircuit polytope $O_M$
\begin{enumerate}
\item[(a)]
is given by the equation $\sum_{i=1}^n x_i = D([n])$ and the inequalities $\sum_{i \in S} x_i \leq D(S) \textrm{ for } S \subseteq [n]$, where $D(S)$ is the number of cocircuits intersecting $S$,
\item[(b)]
has dimension $n-c$ where $c$ is the number of connected components of $M$,
\item[(c)]
has the matroid polytope $P_M = \conv\{e_B \, : \, B \textrm{ basis}\}$ as a Minkowski summand,
\item[(d)]
has at most $r! \cdot b$ vertices, where $b$ is the number of bases. 
\end{enumerate}
\end{reptheorem}

Before we prove this theorem, it is useful to recall some basic facts about generalized permutahedra \cite{Postnikov}. The \emph{permutahedron} $\Pi_n$ is the convex hull of the $n!$ permutations of $\{1, \ldots, n\}$ in $\RR^n$; its normal fan is the braid arrangement formed by the hyperplanes $x_i=x_j$ for $i \neq j$. A \emph{generalized permutahedron} is a polytope $P$ obtained from $\Pi_n$ by moving the vertices (possibly identifying some of them) while preserving the edge directions. This is equivalent to requiring that the normal fan of $P$ is a coarsening of the braid arrangement. \cite{PRW}

Every generalized permutahedron is of the form 
\[
P_n(\{z_I\}) = \{(t_1, \ldots, t_n) \in \RR^n : \sum_{i=1}^n t_i =
z_{[n]}, \sum_{i \in I} t_i \leq z_I \textrm{ for all } I
\subseteq [n]\large\}
\]
where $z_I$ is a minimally chosen real number for each $I \subseteq [n]$, and
$z_{\emptyset}=0$. The vector $(z_I)_{I \subseteq [n]}$ is \emph{submodular}; that is, $z_I + z_J \geq z_{I\cup J} + z_{I\cap J}$ for all subsets $I$ and $J$ of $[n]$. Furthermore, this is a bijection between generalized permutahedra and points in the submodular cone in $\RR^{2^n}$ defined by the submodular inequalities. \cite{AA, Morton, Postnikov, Schrijver}.
This shows that generalized permutahedra are essentially the same as \emph{polymatroids}, which predate them.

There is an alternative description of generalized permutahedra. Every Minkowski sum of simplices of the form $\Delta_I$ is a generalized permutahedron \cite{Postnikov} and, conversely, every generalized permutahedron can be expressed as a \emph{signed} sum of such simplices. \cite{ABD} This automatically implies that $O_M$ is a generalized permutahedron. Also, $P_M$ is the generalized permutahedron $P_n(r(I))_{I \subset [n]}$ where $r(I)$ is the rank of $I$ in the matroid $M$. \cite{ABD, Schrijver}

\begin{proof}[Proof of Theorem \ref{thm:polytope}]
For a polytope $P \in \RR^n$ and a linear functional $w \in (\RR^n)^*$, let $P_w$ be the face of $P$ minimized by $w$. 

\medskip
\noindent
(a) Since $O_M$ is a generalized permutahedron, we have $O_M= P_n({z_I})$ for some vector $z_I$. Notice that $\Delta_D = P_n({z^D_I})$ where $z^D_I$ is $1$ if $I \cap D \neq \emptyset$ and $0$ otherwise. Then the result follows from the fact that $P_n(\{z_I\})+P_n(\{z'_I\}) = P_n(\{z_I+z'_I\})$.

\medskip
\noindent
(b) From the Minkowski sum expression for $O_M$ it is clear that the edge directions of $O_M$ are precisely the edge directions of the various $\Delta_D$. These are the vectors of the form $e_c - e_d$ where $c$ and $d$ are in the same cocircuit; that is, in the same connected component of $M^*$. Their span is the subspace given by the equations $\sum_{i \in K_a} x_i = 0$ for the connected components $K_1, \ldots, K_c$ of $M^*$, which are also the connected components of $M$. The result follows.
 
\medskip
\noindent
(c) When $M$ is the matroid of a linear ideal $I$, this claim is related to (but not implied by) Proposition \ref{prop:Grobnerlinear} and the fact that the matroid polytope is a state polytope of $I$. \cite[Proposition 2.11]{Sturmfels} 
We proceed as follows.

We know that $O_M = P_n(\{D(I)\})$ and $P_{M} = P_n(\{r(I)\})$, where $r$ is the rank function of $M$. We claim that $q(I)=D(I) - r(I)$ is a submodular function; it will then follow that $Q = P_n(\{q(I)\})$ is a generalized permutahedron such that $O_M = P_{M}+Q$. 

Let $\delta_q(S, a, b) = -q(S \cup a \cup b) + q(S \cup a) + q(S \cup b) - q(S)$ for $S \subset [n]$ and $a,b \in [n]-S$; define $\delta_D$ and $\delta_{r}$ analogously. We will prove that $\delta_q$ is always non-negative; this property of ``local submodularity" of $q(I)$ implies its submodularity.

Assume contrariwise that $\delta_q(S, a, b) < 0$. Notice that $\delta_D$ and  $\delta_r$ is non-negative because $D$ and $r$ are submodular, and $\delta_{r}$ equals $0$ or $1$ because $r(S \cup s) - r(S) = 0 \textrm{ or } 1$ for $s \notin S$. Therefore, to have $\delta_q(S, a,  b) = \delta_D(S, a, b) - \delta_{r}(S,a,b)<0$, we must have
\begin{equation}\label{eq:submodular}
\delta_D(S, a, b) = 0,  \qquad \delta_{r}(S, a, b) = 1.
\end{equation}
To have $\delta_{r}(S,a,b) = 1$, we must have, for some s, \begin{equation}\label{eq:ranks}
r(S) = s, \qquad r(S \cup a) = r(S \cup b) = r(S \cup a \cup b) = s+1.
\end{equation}
One easily checks that $\delta_D(S, a, b) = 0$ is the number of cocircuits containing $a$ and $b$ and not intersecting $S$. Since hyperplanes are the complements of cocircuits, it follows that every hyperplane $H \supset S$ must contain either $a$ or $b$. If a hyperplane $H \supset S$ contained one and not the other, say $a \in H$ and $b \notin H$, submodularity would imply 
\[
1 = r(H\cup b) - r(H) \leq r(S \cup a \cup b) - r(S \cup a) = 0
\]
by (\ref{eq:ranks}) and the fact that $H$ is a hyperplane. 
Therefore every hyperplane $H \supset S$ must contain both $a$ and $b$, so every hyperplane of $M/S$ contains both $a$ and $b$. This is only possible if $a$ and $b$ have rank $0$ in $M/S$, which contradicts that $r(S \cup a) = r(S) + 1$.
%
%
%

\medskip
\noindent
(d) Since the normal fan of $O_M$ coarsens the braid arrangement, $\{(O_M)_\pi \, : \, \pi \textrm{ is a permutation of } [n]\}$ is a complete list of the vertices of $O_M$, possibly with repetitions.
%
The $\pi$-minimal vertex is
\[
(O_M)_\pi = 
\sum_{D \textrm{ cocircuit of } M} (\Delta_D)_\pi
= \sum_{D \textrm{ cocircuit of } M} e_{\min_\pi(D)}
= (d^\pi_1, \ldots, d^\pi_n)
\]
where
$d^\pi_i$ is the number of cocircuits of $M$ whose $\pi$-smallest element is $i$.
%

Next we observe that the support of any vertex $(O_M)_\pi$ of $O_M$ is a basis of $M$; more specifically,
\begin{equation}\label{eq:supp}
\textrm{supp} (O_M)_\pi = B_\pi
\end{equation}
where $B=B_\pi$ denotes the \emph{$\pi$-minimal basis} of $M$, which minimizes $\sum_{b \in B} \pi(b)$. This basis is unique by the greedy algorithm for matroids. The claim (\ref{eq:supp}) follows from the following variant of the greedy algorithm for matroids due to Tarjan  called the \emph{blue rule}. \cite[Theorem 2.7]{Kozen} To construct the $\pi$-minimum basis $B_\pi$, we start with all elements of $[n]$ uncolored. We successively choose a cocircuit with no blue elements, and color its smallest element blue. We do this repeatedly, in any order, until it is no longer possible. In the end, the set of blue elements is the basis $B_\pi$. Clearly the blue elements are precisely those $i$ such that $d_i^\pi \neq 0$.

Finally, it remains to observe that for each $\pi$, the vertex $(O_M)_{\pi}$ is determined uniquely by $M$ and the relative order of $\pi(B_\pi)$. To see this, notice that $d^\pi_i$ is the number of cocircuits $D$ of $M$ such that $\pi(i)$ is the smallest element of $\pi(B_\pi \cap D)$. This number only depends on the matroid $M$, the basis $B_\pi$, and the relative order of $\pi(B_\pi)$. Since there are $b$ choices for $B_\pi$ and $r!$ choices for the relative order of $\pi(B_\pi)$, the desired result follows.
\end{proof}

\section{\textsf{External activity complexes and the primary decomposition.}}\label{sec:complex}

Let $M$ be a matroid and let $<$ be a linear order on the ground set $E$. We will build a simplicial complex on the $2|E|$-element set $\{\blue{x_e}, y_e \, : \, e \in E\}$ closely related to the basis activities in $M$. 
Basis activities were originally defined by Tutte (for graphs) and Crapo (for matroids) to give a combinatorial interpretation of the coefficients of the Tutte polynomial, as described in Theorem \ref{thm:Crapo}. \cite{Tutte, Crapo} 
Their clever definition manifests itself algebraically in a very natural way iin our work, thanks to  Theorem \ref{longthm cases}(f) and the following result. 

We identify subsets and monomials, and write, for $A, B \subseteq E$,
\[
\blue{x_A}y_B := \{\blue{x_a} \, : a \in A\} \cup \{y_b \, : \, b \in B\}.
\]


\begin{theorem}\label{thm:simplicialcomplex}
Let $M$ be a matroid on $E$ and let $<$ be a linear order on $E$.
There is a simplicial complex $B_<(M)$ on $\{\blue{x_e}, y_e \, : \, e \in E\}$, called the \emph{external activity complex} of $M$ with respect to $<$, such that
\begin{enumerate}
\item
The facets of $B_<(M)$ are the sets $\blue{x_{B \cup EP(B)}}  y_{B \cup EA(B)}$ for each basis $B$, where $EP(B)$ and $EA(B)$ are the sets of externally passive and externally active elements with respect to $B$.
\item
The minimal non-faces are $\blue{x_{\min C}}  y_{C-\min C}$ for each circuit $C$.
\end{enumerate}
\end{theorem}

\begin{proof}
We need to prove that, for $S, T \subseteq E$
\begin{center}
$\blue{x_S}y_T \subseteq \blue{x_{B \cup EP(B)}}  y_{B \cup EA(B)}$ for some basis $B$ \\
\smallskip
 if and only if \\
$\blue{x_S}y_T \not\supseteq \blue{x_{\min C}}  y_{C-\min C}$ for all circuits $C$.
\end{center}

First we prove the forward direction. Assume, contrariwise, that $\blue{x_S}y_T \subseteq \blue{x_{B \cup EP(B)}}  y_{B \cup EA(B)}$ for some basis $B$ and $\blue{x_S}y_T \supseteq \blue{x_{\min C}} y_{C-\min C}$ for some circuit $C$. Then
\[
\blue{x_{\min C}}  y_{C-\min C} \subseteq \blue{x_{B \cup EP(B)}}  y_{B \cup EA(B)}.
\]
Let $\min C = c$. Since $c \in B \cup EP(B)$, there are two cases:

If $c \in B$: Let $D=D(B,c)$ be the fundamental cocircuit. Then $c \in C \cap D$ and, since $|C \cap D| \neq 1$, we can find another element $d \in C \cap D$. Since $d \in D(B,c)$, we have $c \in C(B,d)$; and $c<d$ because $c = \min C$, so $d$ is not externally active in $B$. Also, $d \in D(B,c)$ implies that $d \notin B$. Therefore $d \notin B \cup EA(B)$. This contradicts that $C - \min C \subseteq B \cup EA(B)$.

If $c \in EP(B)$: We have $c \notin B$. We can find an element $d \in C(B,c)$ with $d<c$. Now $d \in C(B,c)$ implies $c \in D(B,d)=:D$, so $c \in C \cap D$. Again, this means we can find another $e \in C \cap D$. Since $e \in C$ and $c = \min C$, we have $c<e$, and therefore $d<e$. Now, $e \in D$ implies that $e \notin B$.
Also $e \in D(B,d)$ implies that $d \in C(B,e)$; and $d<e$ then implies that $e \notin EA(B)$. Therefore $e \notin B \cup EA(B)$. Again, this contradicts that $C - \min C \subseteq B \cup EA(B)$. This completes the proof of the forward direction.

\medskip

To prove the backward direction, assume that 
$\blue{x_S}y_T \not\subseteq \blue{x_{B \cup EP(B)}}  y_{B \cup EA(B)}$ for all bases $B$. We need to show that 
$\blue{x_S}y_T \supseteq \blue{x_{\min C}}  y_{C-\min C}$ for some circuit $C$.

By Theorem \ref{thm:Crapo} we can write $T=B \cup E - I$ for some basis $B$, some subset $E \subseteq EA(B)$, and some subset $I \subseteq IA(B)$. Then $T \subseteq B \cup EA(B)$, so $S \not\subseteq B \cup EP(B)$. Therefore we can find $s \in S$ with $s \notin B \cup EP(B)$; that is, $s \in EA(B)$. 

Let $C=C(B,s)$. We claim that $\blue{x_S}y_T \supseteq \blue{x_{\min C}}  y_{C-\min C}$. Since $s \in EA(B)$, $s = \min C$, so $S \supseteq \min C$. It remains to show that $T \supseteq C - \min C$. Assume, contrariwise, that $d \in C-\min C$ but $d \notin T$. Since $d \in C-\min C$, $d \in B$. Since $d \notin T = B \cup E - I$, this implies that $d \in I$, so $d$ is internally active in $B$. Therefore $d$ is the smallest element in $D(B,d)$. But $d \in C(B,s)$ implies that $s \in D(B,d)$, which implies that $s > d$. This contradicts that $s = \min C$. The desired result follows.
\end{proof}

\begin{theorem}\label{thm:primarydecomp}
Let $M$ be a matroid on $E$ and let $<$ be a linear order on $E$. The ideal
\[
C(M,<) = \left< \blue{x_{c_1}} y_{c_2} y_{c_3} \cdots y_{c_k} \, : \, C=\{c_1, \ldots, c_k\} \textrm{ is a circuit of $M$ and $c_1 = \min_<C$} \right>
\]
in $\kk[\blue{x_e}, y_e \, : \, e \in E]$ 
is the Stanley-Reisner ideal of the external activity complex $B_<(M)$. Its primary decomposition is
\[
C(M,<) = \bigcap_{B \textrm{ basis of } M} \left< \, \blue{x_e} \, : \, e  \in EA_<(B)\,  , \, y_e \, : \, e \in EP_<(B) \right>.
\]
\end{theorem}

\begin{proof}
By Theorem \ref{thm:simplicialcomplex}.1, $C(M,<)$ is the Stanley-Reisner ideal of $B_<(M)$. Theorem \ref{thm:simplicialcomplex}.2 and \cite[Theorem 1.7]{MS} then imply the desired primary  decomposition.
\end{proof}

The external activity complex and the corresponding monomial ideal are closely related to two important simplicial complexes from matroid theory. If we set $y_i=\blue{x_i}$ we obtain the Stanley-Reisner ideal of the independence complex of $M$, whose facets are the bases of $M$. If we set $\blue{x_i}=1$ we get the Stanley-Reisner ideal of the \emph{broken circuit complex} of $M$, whose facets are the \emph{nbc-bases} of $M$. \cite{Bj}

\section{\textsf{Proofs of our main theorems.}}\label{sec:proofs}

Having built up the necessary combinatorial background, we now use algebraic and geometric tools to complete the proof of our main theorems.

\begin{reptheorem}{longthm cases} Let $L\subset \AA^n$ be a $d$-dimensional linear space and let $\L \subset (\PP^1)^n$ be the closure of $L$ induced by the embedding $\AA^n \subset (\PP^1)^n$. Let $M$ be the matroid of $L$; it has rank $r=n-d$. Then: 
\begin{enumerate}[(a)]
\item The homogenized cocircuits of $I(L)$ minimally generate the ideal $I(\widetilde{L})$.
\item The homogenized cocircuits of $I(L)$ form a universal Gr\"obner basis for $I(\widetilde{L})$, which is reduced under any term order.  
\item The $\mathbb{Z}^n$-multi-degree of $\widetilde{L}$ is
$\sum\limits_{B} t_{b_1}\cdots t_{b_{r}}$ summing over all bases $B = \{b_1,\ldots, b_{r}\}$ of $M$.
\item The bidegree of $\widetilde{L}$ is $t^r h_{M}(s/t)$ where $h_M$ is the $h$-polynomial of $M$.
\item There are at most $r!\cdot b$ distinct initial ideals of $I(\widetilde{L})$, where $b$ is the number of bases of $M$.
\item 
The initial ideal $\initial_< I(\widetilde{L})$ is the Stanley--Reisner ideal of the external activity complex $B_<(M^*)$ of the dual matroid $M^*$. Its primary decomposition is:
$$\initial_< I(\widetilde{L}) = \bigcap_{B \textrm{ basis }} \left< \, \blue{x_e} \, : \, e  \in IA_<(B)\,  , \, y_e \, : \, e \in IP_<(B) \right>$$
where $B = IA_<(B)\,  \sqcup \, IP_<(B)$ is the partition of $B$ into internally active and passive elements with respect to $<$.
\end{enumerate}
\end{reptheorem}

One of our goals is to show that the set $G= \{ f_C^h \}$ of homogenized circuits is a  \emph{universal Gr\"obner basis} (UGB); that is, a Gr\"obner basis for $I(\widetilde{L})$ with respect to any term order.   One key tool is the following:  If two ideals share the same codimension and degree and one contains the other, then under suitably nice conditions we can say they are equal.


\begin{proof}[Proof of (c)] We compute the multi-degree of $\widetilde{L}$ using a geometric argument, recalling the discussion of Section \ref{sec:deg}. For each $\Delta=\{i_1, \ldots, i_r\} \subset [n]$ we let $Z_\Delta = \prod_{i\in \Delta} (\mathbb{P}^1)_i \times \prod_{i\notin \Delta} q_i$, where $q_i$ is a generic point in $(\mathbb{P}^1)_i$ for each $i \notin \Delta$. Then we have
$$\textrm{mdeg } \widetilde{L} = \sum_{\Delta \in {[n] \choose r}} m(Z_\Delta, \widetilde{L}) t_{i_1}\cdots t_{i_r}.$$
where $m(Z_\Delta, \widetilde{L})$ is the number of points in $Z_\Delta \cap \widetilde{L}$ counted with multiplicity. 
We will prove that 
%
%
\begin{equation*}\label{eq:multidegree}
m(Z_\Delta,\widetilde{L}) = \begin{cases}
1 &  \text{if } \Delta \text{ is a basis of $M$, and}\\
0 & \textrm{otherwise,}
\end{cases}
\end{equation*}
from which our formula for $\textrm{mdeg } \L$ will follow.


First let $\Delta$ be a basis., 
Since the points $q_i$ in $Z_\Delta$ are general, we may suppose that $y_i \neq 0$ for $i \notin \Delta$. Now let $i \in \Delta$.
Since $\Delta$ is a basis, there is a cocircuit $D$ containing $i$ whose support is contained in $([n]-\Delta)\cup i$; 
in fact, this is the fundamental cocircuit $D(\Delta, i)$.  If $y_i$ were equal to zero, then the homogenized cocircuit equation $f_D^h = 0$ would force $\blue{x_i} = 0$, which is impossible.  Hence all intersections must occur in the affine patch where no coordinate $y_i$ equals zero; but then we are working in the original affine space, so it is clear that $Z_\Delta \cap \widetilde{L} = Z_\Delta \cap L$ is a single point with no multiplicity. Therefore $m(Z_\Delta, \widetilde{L}) = 1$.

On the other hand, if $\Delta = \{i_1,\ldots,i_r\}$ is not a basis, then there is a cocircuit $D$ which is disjoint from $\Delta$.  This means that $Z_\Delta$ does not meet the hypersurface defined by the homogenized cocircuit $f_{D}^h$.  Hence $\widetilde L$ does not meet $Z_\Delta$ and $m(Z_\Delta, \widetilde{L}) =0$. The desired result follows.
\end{proof}

\begin{proof}[Proof of (b)] 
Let $G=\{f_D^h \, : \, D \textrm{ cocircuit of } M\} \subset I(\widetilde{L})$ denote the set of homogenized cocircuits. 
Let $<$ be any monomial term order on $\kk[\blue{x_1},\ldots,\blue{x_n}, y_1, \ldots, y_n]$, and let $\initial_<G$ denote the ideal generated by the leading terms of the polynomials in $G$.
 We need to show that $\initial_<G = \initial_<I(\widetilde{L})$.
We begin with a remark:
\begin{remark}
If $<$ is any monomial term order, it is sufficient for Gr\"obner computations to assume that $<$ is given by a weight order $w$ on the $2n$ variables  $\blue{x_1},\ldots,\blue{x_n}, y_1,\ldots, y_n$. \cite[Prop. 1.11]{Sturmfels} Since all of the polynomials in $I(\widetilde{L})$ are multi-homogeneous, the term order $w'$ given by 
$$\begin{array}{rcl}w'(\blue{x_i}) &=& w(\blue{x_i}) - w(y_i) \\ w'(y_i) &= &0\end{array}$$
will pick out the same initial terms on $I(\widetilde{L})$ as $w$. Thus we may assume that the weights on the $y$ variables are all zero. The resulting linear order on $1, \ldots, n$ is the reverse\footnote{We reverse it because the initial terms $\initial_<f$ is the largest monomial of $f$, while basis activities are usually defined in terms of the smallest elements of circuits and cocircuits.} of the  linear order that we imposed on $[n]$ in Remark \ref{rem:order}. 
\end{remark}

Notice that each term of a given $f_D^h$ has degree one in the $y$-variables and is homogeneous.  Thus by the remark, the leading term of $f_D^h$ depends only on the linear order on $1, \ldots, n$. Therefore 
\[
\initial_<G = \left< \blue{x_{d_1}} y_{d_2} y_{d_3} \cdots y_{d_k} \, : \, D=\{d_1, \ldots, d_k\} \textrm{ is a cocircuit of $M$ and $d_1 = \min_<D$} \right>.
\]
In other words, 
\[
\initial_<G = C(M^*,<)
\]
is the Stanley-Reisner ideal of the external activity complex $B_<(M^*)$ of $M^*$, as described 
in Theorem \ref{thm:primarydecomp}. Therefore
\begin{equation}\label{eq:D}
\initial_<G = \bigcap_{B \textrm{ basis of } M} \left< \, \blue{x_e} \, : \, e  \in IA_<(B)\,  , \, y_e \, : \, e \in IP_<(B) \right>.
\end{equation}
Applying Proposition \ref{prop:mdeg} to (\ref{eq:D}) then gives
\[
\textrm{mdeg } \initial_<G = \sum\limits_{B \textrm{ basis}} t_{b_1}\cdots t_{b_{r}} = \textrm{mdeg } I(\widetilde{L}). 
\]
We also have that 
\[
 \textrm{mdeg } I(\widetilde{L}) =  \textrm{mdeg } \initial_< I(\widetilde{L}). 
\]
since multi-degree is preserved by flat degenerations. Therefore both ideals in the inclusion
$$\initial_< G \subset \initial_< I(\widetilde{L})$$
 have the same multidegree.   Since the smaller ideal is reduced and equidimensional, this implies that 
 $\initial_< G = \initial_< I(\widetilde{L})$.  \cite[Lemma 1.7.5]{KM}.  Since $<$ was arbitrary, $G=\{f_C^h\}$ is a universal Gr\"obner basis for $I(\widetilde{L})$.    To see that $G$ is reduced for each term order, just notice that no term divides another, because no cocircuit contains another. 
\end{proof}

\begin{proof}[Proof of (f)] 
Now that we know that  $\initial_< G = \initial_< I(\widetilde{L})$, this follows from (\ref{eq:D}).
\end{proof}

\begin{proof}[Proof of (d)] By (f) and Theorem \ref{thm:Crapo}, any initial ideal has bidegree $t^rh_M(s,t)$. The result then follows from  the fact that bidegree is degenerative. \cite{MS}
\end{proof}

\begin{proof}[Proof of (a)] 
Since no term of any generator in $G$ divides any other term, the elements of $G$ are linearly independent over $\k$, and they minimally generate $\left<G\right> = I(\widetilde{L})$. 
\end{proof}

\begin{proof}[Proof of (e)] 
Finally we prove our upper bound for the number of distinct initial ideals of $\initial_< I(\widetilde{L})$. One way to proceed is to invoke \cite[Cor. 2.9]{Sturmfels}: if $G$ is a universal Gr\"obner bases of $I$ which is a reduced Gr\"obner bases with respect to any term order $<$, then the Minkowski sum $\sum_{g \in G} \textrm{New}(g)$ is a state polytope for $I$, so its vertices are in bijection with the initial ideals of $I$. Here $\textrm{New}(g)$ denotes the Newton polytope of $g$. The Newton polytope of each homogenized cocircuit $f_D^h$ is, after translation, equal to $\nabla_I = \conv\{h_i \, : \, i \in I\}$, where $h_d=f_d-g_d$ and $f_1, \ldots, f_n, g_1, \ldots, g_n$ is the standard basis for $\RR^{2n}$. It follows that $\sum_{f_D^h \in G} \textrm{New}(g) =  \sum_{D \textrm{ cocircuit}} \nabla_D$ is affinely isomorphic to $\sum_{D \textrm{ cocircuit}} \Delta_D = O_M$. The result then follows by Theorem \ref{thm:polytope}(d).

We also give a self-contained algebraic proof. If we take an initial ideal of $I(\widetilde{L})$ and set all the $y$-variables equal to $1$, then we obtain an initial ideal of $I(L)$.  Hence we have a map
$$\{\textrm{initial ideals of } I(\widetilde{L})\} \stackrel{\rho}{\to} \{\textrm{initial ideals of } I(L)\}.$$
This map is surjective since the cocircuits and their homogenizations are universal Gr\"obner bases. 

Each initial ideal $\initial_< I(\widetilde{L})$ of $I(L)$ is of the form $I_B = \langle x_{b_1}, \ldots, x_{b_r} \rangle$, where $B=\{b_1,\ldots, b_r\}$ is the $<$-minimum basis of $B$. \cite[Prop. 2.11]{Sturmfels}    Hence the map $\rho$ is surjective onto a set with $b$ elements.  The pre-image of $I_B$ is a set of ideals in the $x$ and $y$ variables.  Now, any term order which determines an ideal in $\rho^{-1}(I_B)$ obviously always selects a term from each homogenized cocircuit with an $x_{b_i}$ for some $i$.  Hence, the relative order of the variables $x_{b_i}$ is sufficient to determine $\initial_< I(\widetilde{L})$.  There are $r!$ such orders, so the number of initial ideals of $I(\widetilde{L})$ is at most $r! \cdot b$.
\end{proof}

\begin{reptheorem}{thm:bettinumbers}
Let $L$ be a linear $d$-space in $\AA^n$, and $I(\widetilde{L})$ the ideal of its closure in $(\mathbb{P}^1)^n$. 
The non-zero multigraded Betti numbers of $S/I(\widetilde{L})$ 
are precisely:
\[
\beta_{i,\a} (S/I(\widetilde{L})) = 
|\mu(F, \widehat{1})|
\]
for each flat $F$ of $M$, where $i=r-r(F)$, and $\a = e_{[n]-F}$.
Here $\mu$ is the M\"obius function of the lattice of flats of $M$.  Furthermore, all of the initial ideals have the same Betti numbers: 
$$\beta_{i,\a} (S/I(\widetilde{L})) = \beta_{i,\a} (S/(\initial_< I (\widetilde{L})))$$
for all $\a$ and for every term order $<$.  
\end{reptheorem}

\begin{proof}
As we already remarked, the initial ideal
\[
\initial_< I(\L) = C(M^*,<) = \left< \blue{x_{d_1}} y_{d_2} y_{d_3} \cdots y_{d_k} \, : \, D=\{d_1, \ldots, d_k\} \textrm{ is a cocircuit of $M$ and $d_1 = \min_<D$} \right>
\]
is closely related to the Stanley-Reisner ideal
\[
I_{IN(M^*)} = \left< x_{d_1} x_{d_2}  \cdots x_{d_k} \, : \, D=\{d_1, \ldots, d_k\} \textrm{ is a cocircuit of $M$}\right>
\]
of the independence complex $IN(M^*)$ of the dual matroid $M^*$. More precisely, the second is obtained from the first by setting $y_i=\blue{x_i}$. In fact, we now show that this substitution is equivalent to taking $C(M^*,<)$ modulo a regular sequence.  This will follow from the primary decomposition of $\initial_< I(\L)$ given by Theorem \ref{thm:primarydecomp}, together with the following lemma:

\begin{lemma}\label{lemma:regsequence}
Let $I$ be a squarefree monomial ideal in $S = \kk[\blue{x_1},\ldots, \blue{x_n}, y_1, \ldots, y_n]$ satisfying
\begin{itemize}
\item[(P1)] For each $i$, no associated prime of $I$ contains both $\blue{x_i}$ and $y_i$, and 
\item[(P2)] No minimal generator of $I$ contains a product of the form $\blue{x_i}y_i$.  
\end{itemize}
Then
$$\{\blue{x_1} - y_1, \ldots, \blue{x_n} - y_n\}$$ is a regular sequence on $S/I$.
\end{lemma}

\begin{proof}
Notice that (P1) implies that $\blue{x_1} - y_1$ is a regular element on $S/I$.  We now form the ideal  
$$I' = I\otimes S/(\blue{x_1} -y_1) $$
which we realize as an ideal in the polynomial ring $S/(\blue{x_1})$ via the substitution $\blue{x_1} \mapsto y_1$.  We claim that $I'$ has properties $(P1), (P2)$ and then the proof will be complete by induction.  

First, the minimal generators of $I'$ are precisely the generators of $I$ after the substitution $\blue{x_1} \mapsto y_1$.  Thus (P2) is satisfied.  

Now denote the primary decomposition of $I$ as $ I =  \bigcap P_i.$  Let $P'_i$ denote the ideal obtained from $P_i$ after the substitution $\blue{x_1} \mapsto y_1$.  We claim that 
$$I' = \bigcap P'_i.$$
Substitution is a ring map, and this easily implies the inclusion $I'\subset \bigcap P'_i$. For the opposite conclusion, suppose that $f$ is a minimal generator of $\bigcap P'_i$.  We need to show that $f \in I'$. Notice that $f$ does not involve $\blue{x_1}$.  We have two cases:

Case 1:  $y_1$ does not divide $f$.   In this case, $f \in P'_i$ implies $f \in P_i$ for all $i$, so $f \in I$. But then $f \in I'$ also, 
 since $f$ does not involve $\blue{x_1}$ or $y_1$, the variables that change under our substitution.

Case 2:  $y_1$ divides $f$, say $f=y_1 g$.  Consider the element $h=\blue{x_1}f$.  Since $h$ is divisible by both $\blue{x_1}$ and $y_1$, and since $f$ is in $\bigcap P'_i$, we know $h$ is in fact in each ideal $P_i$.  Thus $h=\blue{x_1} y_1 g\in I$.  But since $I$ has no minimal generators by (P1) divisible by $\blue{x_1}y_1$ we know that either $\blue{x_1}g$ or $y_1g$ must be in $I$.  Under the substitution, both of these elements will be sent to $f$, so that $f\in I'$. 

We conclude that indeed $I' = \bigcap P'_i$, which implies that $I'$ satisfies (P1). This completes the proof by induction.
\end{proof}

With Lemma \ref{lemma:regsequence} at hand, we are now ready to prove Theorem \ref{thm:bettinumbers}. As taking initial ideals is a flat degeneration, we have
\[
\beta_{i,\a} (S/I(\widetilde{L})) \leq \beta_{i,\a} (S/(\initial_< I (\widetilde{L}))).
\]
for all $i$ and $\a$. As explained in Section \ref{sec:freeres}, the only way that this inequality can be strict is due to a consecutive cancellation in the minimal free resolution of $S/\initial_< I(\widetilde{L})$. This requires that $\beta_{i, \a}(S/(\initial_< I(\widetilde{L})))$ is nonzero for some $\a$ and for two consecutive values of $i$.
However, since $\{\blue{x_1}-y_1, \ldots, \blue{x_n}-y_n\}$ is a regular sequence by the previous Lemma, we know that the Betti numbers of $S/(\initial_< I (\widetilde{L}))$ equal those of the independence ideal $\k[x_1, \ldots, x_n]/I_{IN(M^*)}$.   Theorem \ref{thm:onlynonzeroBetti} then shows that no such consecutive cancellations are possible.
\end{proof}

\begin{reptheorem}{CM}
If $L$ is a linear space the ideal $I(\widetilde{L})$ and all of its initial ideals are Cohen-Macaulay. 
\end{reptheorem}

\begin{proof}
An ideal is Cohen-Macaulay if an only if its codimension is equal to its projective dimension.  Since both ideals are of the same codimension and 
\[
\beta_{i,\a} (S/I(\widetilde{L})) = \beta_{i,\a} (S/(\initial_< I (\widetilde{L})))
\]
by Theorem \ref{thm:bettinumbers}, it is sufficient to prove that $\initial_< I (\widetilde{L})$ is Cohen-Macaulay. 

Now, Theorem \ref{thm:bettinumbers} also tells us that the projective dimension of $S/(\initial_< I (\widetilde{L}))$ equals $r$, the rank of $M$. Also since $S/(\initial_< I (\widetilde{L}))$ is the Stanley-Reisner ring of $B_<(M^*)$, whose facets have $2n-r$ elements, its codimension is also $r$. The desired result follows.
%
%
\end{proof}


\section{\textsf{The non-homogeneous case: affine linear spaces.}}\label{sec:affine}
So far, we have assumed that the linear space $L$ was actually a vector subspace of $\kk^n$.  This is a minor assumption, but nonetheless, the nonhomogeneous case has some interesting features. 

In this section, suppose that $L$ is an \textbf{affine} linear subspace defined by the matrix equation $$A\cdot \vec x = \vec b,$$ and let $\widetilde{L}$ be its closure in $(\PP^1)^n$. Now the invariants of $\widetilde{L}$ are controlled by (any two of) the following triple of matroids $(M_{hom}, M, M')$: 

$\bullet$ the matroid $M$ on $[n]$ that we associated to the subspace $A \cdot \vec x = \vec 0$,

$\bullet$ the \emph{augmented matroid} $M_{hom}$ on $[0,n]$ associated to the 
subspace $(A \, | \, (-b)) \cdot \overrightarrow{ (x, x_0) }= \vec 0$,

$\bullet$ the matroid $M'$ on $[n]$ of the subspace obtained from $(A \,  |\,   (-b)) \cdot \overrightarrow{ (x, x_0) }= \vec 0$ by eliminating $x_0$.

\noindent Any two of these matroids determine the third. They are related by 
\[
M = M_{hom} \backslash 0, \qquad M'=M_{hom}/0.
\]
The triple $(M_{hom}, M, M')$ is equivalent to a  \emph{pointed matroid} \cite{Brylawski} or a \emph{semimatroid} \cite{Ar07}.

\subsection{\textsf{Matroid preliminaries: morphisms and Tutte polynomials.}}

The matroids $M$ and $M'$ above can be thought to have the same ground set. They constitute a \emph{morphism of matroids}, denoted $M \rightarrow M'$; this means that   
every flat of $M'$ is a flat of $M$. 
Just as matroids are an abstraction of vector configurations, matroid morphisms are an abstraction of linear maps.


LasVergnas \cite{LasVergnas80} defined the \emph{Tutte polynomial of a morphism $M \rightarrow M'$} to be
\[
T_{M \rightarrow M'}(x,y,z) = \sum_{S \subseteq E} (x-1)^{r'-r'(S)} (y-1)^{|S|-r(S)} z^{rcd_{M,M'}(S)}
\]
where $r$ and $r'$ are the rank functions of $M$ and $M'$, and $rcd_{M,M'}(S) = (r-r') - (r(S) - r'(S))$. 

He also gave an activity interpretation of this polynomial, which we now describe.
For an independent set $X$ of $M$ and $i \notin X$, the set $X \cup i$ contains at most one circuit $C$ of $M$ (which must contain $i$). If $C$ does exist and $i$ is the smallest element in $C$, then we say $i$ is \emph{externally active} with respect to $X$ in $M$.
Dually, for a spanning set $X$ of $M'$, the set $(E-X) \cup i$ contains at most one cocircuit $D$ of $M'$ (which must contain $i$). If $D$ does exist, and $i$ is the smallest element in $D$, then we say $i$ is \emph{internally active} with respect to $X$ in $M'$.

\begin{theorem}\cite{LasVergnas80}\label{th:morphism}
Consider any matroid morphism $M \rightarrow M'$ and any linear order $<$ on the ground set $E$ of $M$ and $M'$. Then
\[
T_{M \rightarrow M'}(x,y,z) = \sum_{S \subseteq E} x^{|IA'(S)|} y^{|EA(S)|} z^{rcd_{M,M'}(S)}
\]
summing over the sets $S$ which are spanning in $M'$ and independent in $M$, where $IA'(S)$ represents the set of internally active elements of $S$ in $M'$, and $EA(S)$ represents the externally active elements with respect to $S$ in $M$. 
\end{theorem}

%
%

\subsection{\textsf{A non-homogeneous example.}}


Before we state and prove our theorems about affine subspaces, we carry out an example in detail which displays most of the interesting features.

\begin{example}\label{non-hom ex}
Consider the affine subspace $L$ of $\AA^6$ given by the linear ideal
\[
I(L)= \left< x_1+x_2+x_6 +a , \,\, x_2-x_3+x_5 + b, \,\, x_3+x_4 + c\right>,
\]
where $a,b,c$ are parameters. 
\end{example}

The matroid $M$ is the same one of Example \ref{ex}, while $M_{hom}$ is the matroid of the ideal 
\[
I(L_{hom})= \left< x_1+x_2+x_6 +a x_0 , \,\, x_2-x_3+x_5 + b x_0, \,\, x_3+x_4 + c x_0\right>.
\]
in seven variables, defining a linear space $L_{hom}$ in $\AA^7$. We can also see $M_{hom}$ as the matroid obtained by adding the point $(a,b,c)$ to our point configuration of columns. 

From now on we assume  $(a,b,c)=(1,0,1)$. Figure \ref{fig:points2} shows the enlarged point configuration, the original point configuration, and the contracted point configuration. They correspond, respectively, to the matroids $M_{hom}$, $M$, and $M'$.

\begin{figure}[ht]
 \begin{center}
  \includegraphics[scale=1]{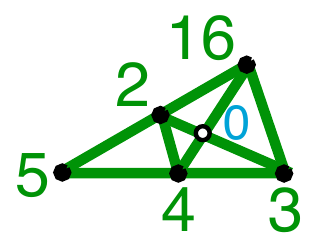}
  \qquad \qquad \qquad 
  \includegraphics[scale=1]{dualvectorconfig} 
  \qquad \qquad \qquad \quad
  \includegraphics[scale=1]{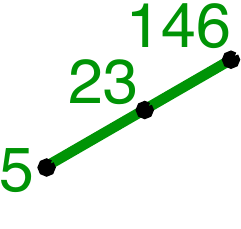}
  \caption{ \label{fig:points2} 
The triple $(M_{hom}, M, M')$ corresponding to the affine subspace $L$.
}
 \end{center}
\end{figure}

Notice that every basis of $M$ is still a basis in $M_{hom}$. Every cocircuit of $M$ gives rise to a cocircuit of $M_{hom}$ by adding $0$ if necessary. As a partial converse, every cocircuit of $M_{hom}$ contains a cocircuit of $M$. Furthermore, if $D$ is a cocircuit of $M_{hom}$ containing $0$, then $D - 0$ is a cocircuit of $M$.
In our example the cocircuits are:
\begin{eqnarray*}
\D &=& \{126,\,\, 1356,\,\, 1456,\,\, 235,\,\, 245, \,\,34\} \\
\D_{hom} &=& \{0126, 01356, 1456, 235, 0245, 034, 12346\} 
\end{eqnarray*}

We will show in Theorem \ref{affinelongthm} that the six \emph{homogenized cocircuits} of $I(L)$, which correspond to the cocircuits $\D$ of $M$, minimally generate $I(\widetilde{L})$ and give a universal Gr\"obner basis:
\[
I(\widetilde{L}) = \left< \bblue{x_1}y_2y_6+y_1\bblue{x_2}y_6+y_1y_2\bblue{x_6}+1y_1y_2y_6, \,\, 
 \ldots,\,
\bblue{x_3}y_4+y_3\bblue{x_4}+1y_3y_4
\right>.
\]
Some of the invariants of $I(\widetilde{L})$ depend only on $M$ as before.
The multidegree of $\widetilde{L}$ is still given by the thirteen bases of the matroid $M$. The multigraded Betti numbers also stay the same as before. 

On the other hand, the initial ideals of $I(\widetilde{L})$ depend on the augmented matroid $M_{hom}$ as well. For  $w(\bblue{x_1}/y_1)> \cdots> w(\bblue{x_6}/y_6) > 0$, the initial ideal $\initial_< I(\widetilde{L})$ is the same as in the homogeneous case:
\begin{eqnarray*}
\initial_< I(\widetilde{L}) &=& \left<\bblue{x_1}y_2y_6, \,\, \bblue{x_1}y_3y_5y_6, \,\,\bblue{x_1}y_4y_5y_6, \,\, \bblue{x_2}y_3y_5, \,\,\bblue{x_2}y_4y_5, \,\,\bblue{x_3}y_4
\right>
\\\
&=& 
\left<\bblue{x_1},\bblue{x_2},\bblue{x_3}\right> \, \cap \, 
\left<\bblue{x_1},\bblue{x_2},y_4\right> \, \cap \, 
\left<\bblue{x_1},y_3,y_4\right> \, \cap \, 
\left<\bblue{x_1},\bblue{x_3},y_5\right> \, \cap \, 
\left<\bblue{x_1},y_4,y_5\right> \, \cap \, \\
&&
\left<y_2,y_3,y_4\right> \, \cap \, 
\left<y_2,\bblue{x_3},y_5\right> \, \cap \, 
\left<x_2,\bblue{x_3},y_6\right> \, \cap \, 
\left<y_2,y_4,y_5\right> \, \cap \, 
\left<\bblue{x_2},y_4,y_6\right> \, \cap \, \\
&&
\left<y_3,y_4,y_6\right> \, \cap \, 
\left<\bblue{x_3},y_5,y_6\right> \, \cap \, 
\left<y_4,y_5,y_6\right>.
\end{eqnarray*}
However, if $0>w(\bblue{x_1}/y_1)> \cdots > w(\bblue{x_6}/y_6)$, we have
\begin{eqnarray*}
\initial_< I(\widetilde{L}) 
&=& \left<y_1y_2y_6, \,\, y_1y_3y_5y_6, \,\,\bblue{x_1}y_4y_5y_6, \,\, \bblue{x_2}y_3y_5, \,\, y_2y_4y_5, \,\, y_3y_4
\right>
\\\
&=& 
\left<\bblue{x_1},y_2, y_3\right> \, \cap \, 
\left<y_1,\bblue{x_2},y_4\right> \, \cap \, 
\left<y_1,y_3,y_4\right> \, \cap \, 
\left<y_1,y_3,y_5\right> \, \cap \, 
\left<y_1,y_4,y_5\right> \, \cap \, \\
&&
\left<y_2,y_3,y_4\right> \, \cap \, 
\left<y_2,y_3,y_5\right> \, \cap \, 
\left<y_2,y_3,y_6\right> \, \cap \, 
\left<y_2,y_4,y_5\right> \, \cap \, 
\left<\bblue{x_2},y_4,y_6\right> \, \cap \, \\
&&
\left<y_3,y_4,y_6\right> \, \cap \, 
\left<y_3,y_5,y_6\right> \, \cap \, 
\left<y_4,y_5,y_6\right>.
\end{eqnarray*}
We will see that the 13 primary components still correspond to the 13 bases of $M$. However, in the primary component $\left<z_i \, : \, i \in B \right>$, we have $z_i =  \bblue{x_i}$ if $i$ is internally active in $B$ \textbf{as a basis of $M_{hom}$}, and $z_i=y_i$ otherwise.

Notice that, in contrast with the linear case, $I(\widetilde{L})$ is no longer bihomogeneous under the bigrading $\bideg \bblue{x_i} = (1,0)$ and $\bideg y_i = (0,1)$. However, some initial ideals still have interesting bidegrees. 
For any term order with $w(\bblue{x_1}/y_1), \cdots, w(\bblue{x_6}/y_6)>0$, we saw in Example \ref{ex} that the bidegree of $\initial_< I(\widetilde{L})$ is $s^3+3s^2t+5st^2+4t^3$. This is essentially the $h$-polynomial of $M$.
We will also show that for any term order with $0>w(\bblue{x_1}/y_1), \cdots, w(\bblue{x_6}/y_6)$, the bidegree of $\initial_< I(\widetilde{L})$ is $3st^2+10t^3$. It is not obvious that all these initial ideals should have the same bidegree; this will follow from the fact that this polynomial is an evaluation of the Tutte polynomial of the matroid morphism $M \rightarrow M'$.

The number of initial ideals also depends on $M_{hom}$.
%
Table \ref{table} shows these numbers for five choices of $(a,b,c)$. In Figure  \ref{fig:points} they correspond, respectively, to adding point $0$ as a loop, as the intersection of lines $23$ and $146$, as a generic point on line $24$, as a generic point on line $136$, or as a generic point in the plane. Somewhat surprisingly, a special choice of $(a,b,c)$ can lead to more initial ideals for $I(\widetilde{L})$ than a generic choice.

\begin{table}[h]
\centering
\begin{tabular}{|c|cc|}
\hline
$(a,b,c)$ &  number of initial \,\,\,\, & number of initial  \\
& ideals of $I(\widetilde{L})$\,\,\,\, & ideals of $I(\widetilde{L}_{hom})$ \\
\hline 
$(0,0,0)$ &  72 & 72 \\
$(1,0,1)$ & 124 & 144 \\
$(2,2,3)$ & 114 & 156\\
$(1,-1,1)$ & 111 & 150 \\
$(1,2,3)$ & 107 & 162 \\
\hline
\end{tabular}
\caption{Number of initial ideals for various choices of $(a,b,c)$\label{table}}
\label{table of inits}
\end{table}

For homogeneous linear spaces $L$, we proved that the number of initial ideals of $I(\widetilde{L})$ is at most $r! \cdot b$ where $r=n-d$ is the codimension of $L$ and $b$ is the number of bases of $M(L)$. This bound is visibly false in the non-homogeneous case, as shown in Table \ref{table of inits}. Instead, we will prove a bound of $r! \cdot b_{hom}$, where  $b_{hom}$ is the number of bases of $M_{hom}$.


The following theorem is the affine analog of Theorem \ref{longthm cases}.

\begin{theorem}\label{affinelongthm} Let $L\subset \AA^n$ be a $d$-dimensional \textbf{affine} space and let $\L \subset (\PP^1)^n$ be the closure of $L$ induced by the embedding $\AA^n \hookrightarrow (\PP^1)^n$. Let $(M_{hom}, M, M')$ be the triple of matroids of $L$.
%
Then: 
\begin{enumerate}[(a)]
\item The homogenized cocircuits of $I(L)$ minimally generate the ideal $I(\widetilde{L})$.
\item The homogenized cocircuits of $I(L)$ form a universal Gr\"obner basis for $I(\widetilde{L})$, which is reduced under any term order.  
\item The $\mathbb{Z}^n$-multi-degree of $\widetilde{L}$ is
$\sum\limits_{B} t_{b_1}\cdots t_{b_{r}}$ summing over all bases $B = \{b_1,\ldots, b_{r}\}$ of $M$.
\item The bidegree of the ideal $I(\widetilde{L})$ is not well-defined unless $L$ is a linear subspace. However: 

1. For every term order $<$ with $x_i > y_i$ for all $i$,
\[
\bideg \initial_< I(\widetilde{L}) = t^r h_{M}(s/t),
\]
where $h_M(x) = T_M(x,1)$ is the $h$-polynomial of $M$ and $T_M(x,y)$ is its Tutte polynomial.

2. For every term order $<$ with $x_i < y_i$ for all $i$,
\[
\bideg \initial_< I(\widetilde{L}) = t^rT_{M \rightarrow M'}(s/t,1,0),
\]
where 
$T_{M \rightarrow M'}(x,y,z)$ is the Tutte polynomial of the matroid morphism $M \rightarrow M'$.
\item There are at most $r!\cdot b_{hom}$ distinct initial ideals of $I(\widetilde{L})$, where $b_{hom}$ is the number of bases of $M_{hom}$.
\item 
The primary decomposition of an initial ideal $\initial_< I(\widetilde{L})$ is given by:
$$\initial_< I(\widetilde{L}) = \bigcap_{B \textrm{ basis of } M} \left< \, \blue{x_e} \, : \, e  \in IA^{hom}_<(B)\,  , \, y_e \, : \, e \in IP^{hom}_<(B) \right>$$
where $B = IA^{hom}_<(B)\,  \sqcup \, IP^{hom}_<(B)$ is the partition of $B$ into internally active and passive elements with respect to $<$, \emph{when regarded as a basis of $M_{hom}$}.
\end{enumerate}
\end{theorem}

\begin{remark}\label{rem:order}
Again, a remark is in order about the choice of order in Theorem \ref{longthm cases}(d,f). Now an initial ideal $\initial_< I(\widetilde{L})$ is determined by the relative order of $0$ and the weights of $\blue{x_1}/y_1, \ldots, \blue{x_n}/y_n$. We then assign the opposite order $<$ to $0, 1, \ldots, n$ in $M$ and $M_{hom}$, and this is the linear order $<$ with respect to which $IA_<^{hom}(B)$ and $IP_<^{hom}(B)$ are defined.
\end{remark}

\begin{proof}[Proof of (c).]
The proof of Theorem \ref{longthm cases}(c) carries through unchanged to show that 
$$\mathrm{mdeg}\  \widetilde{L} =  \sum_{b\in B} t_{b_1}\cdots t_{b_k}$$
where the sum is taken over all bases $b = \{b_1,\ldots, b_k\}$ of $M(L)$. 
\end{proof}

\begin{proof}[Proof of (a).] Let
$$\mathcal{D}^h  = \{ f_D^h \ | \ D \in \mathcal{D} \}$$
be the set of homogenized cocircuits of $I(L)$. To see that it is a minimal generating set for $I(\widetilde{L})$, we notice that under the lexicographic monomial order
$\blue{x_1}>\cdots >\blue{x_n} > y_1 > \cdots > y_n$, the initial terms of each
element of $\mathcal{D}^h$ are independent of $\vec b$.  In fact,
they are the same as the leading terms in the case when $\vec b = 0$.
Hence the ideal these monomials generate has a primary decomposition
given by Theorem \ref{longthm cases}(f).  Then the argument of Theorem \ref{longthm cases}(b) shows that $\mathcal{D}^h$ is a Gr\"obner basis under \textbf{this} term order, and in particular it generates $I(\widetilde{L})$. The argument of Theorem \ref{longthm cases}(a) then shows that $\mathcal{D}^h$
is indeed a minimal generating set for $I(\widetilde{L})$. 
\end{proof}

\begin{proof}[Proof of (f).]
Let $<$ be a monomial term order on $\kk[\blue{x_1},\ldots,\blue{x_n},y_1,\ldots,y_n]$. Say $<$ is given by a weight vector on $\blue{x_1}, \ldots, \blue{x_n}, y_1, \ldots, y_n$. Redefining the weights to be $w'(\blue{x_i}) = w(\blue{x_i}) - w(y_i)$ and $w'(y_i)=0$ for all $i$ does not affect the leading terms of polynomials in $I(\widetilde{L})$. This allows us to assume that
$$
w(\blue{x_1}) > w(\blue{x_2})> \cdots >w(\blue{x_n}), \qquad w(y_i) = 0 \mbox{ for all } i.$$
We extend $<$ to a term order on $\kk[\blue{x_0},\ldots,\blue{x_n},y_0,\ldots,y_n]$ by assigning $w(\blue{x_0}) = w(y_0) = 0.$  This  ensures that the initial term computations in $I(\widetilde{L}_{hom})$ mimic exactly those in $I(\widetilde{L})$.



Now let $\mathcal{D}_{hom}^h  = \{ g_D^h \ | \ D \in \mathcal{D}_{hom} \}$ be the set of homogenized cocircuits of $I(\widetilde{L}_{hom})$. This is a universal Gr\"obner basis for $I(\widetilde{L}_{hom})$ by Theorem \ref{longthm cases}(b); let
\[
J:=\initial_< \mathcal{D}^h_{hom} = \initial_< I(\widetilde{L}_{hom}).
\]
We claim that
\begin{equation}\label{eq:local}
\initial_< I(\widetilde{L}) = J(\blue{x_0} = 1, y_0 = 1).
\end{equation}

First notice that any $f \in I(\widetilde{L})$ can be further ``bi-homogenized" to a polynomial $f' \in I(\widetilde{L}_{hom})$ which is homogeneous in the $x$ variables and in the $y$ variables, by multiplying each monomial by a suitable factor of $\blue{x_0}$ or $y_0$. Then one easily checks that 
$\initial_< f  = \initial_< f' |_{\blue{x_0} = y_0 = 1}$. This shows that $\initial_< I(\widetilde{L}) \subset J(\blue{x_0} = 1, y_0 = 1).$

To show the other inclusion it suffices to show that that $J(\blue{x_0} = 1, y_0 = 1)$ and $I(\widetilde{L})$ have the same multidegree, and we can do that using Proposition \ref{prop:mdeg}. 
The primary decomposition of $J$ has components corresponding to the bases of $M_{hom}$, as described in Theorem \ref{longthm cases}(f). In this primary decomposition, setting $\blue{x_0} = y_0 = 1$ is equivalent to ignoring the components that contain $\blue{x_0}$ or $y_0$, which  correspond to the bases of $M_{hom}$ that contain $0$.  Thus the only components that survive are those that correspond to bases of $M$, and 
\begin{equation}\label{eq:in-nonhom}
J(\blue{x_0 = 1}, y_0 = 1) = \bigcap_{B \textrm{ basis of } M} \left< \, \blue{x_e} \, : \, e  \in IA^{hom}_<(B)\,  , \, y_e \, : \, e \in IP^{hom}_<(B) \right>
\end{equation}
where $B = IA^{hom}_<(B)\,  \sqcup \, IP^{hom}_<(B)$ is the partition of $B$ into internally active and passive elements with respect to $<$, when regarded as a basis of $M_{hom}$. It follows that $\textrm{mdeg } J = \sum_B t_{b_1}\cdots t_{b_r}$ where we sum over all bases $B = \{b_1, \ldots, b_r\}$ of $M$. By (c), this is equal to $\textrm{mdeg } \initial_< I(\widetilde{L}) = \textrm{mdeg } I(\widetilde{L})$. This completes the proof of (\ref{eq:local}), and combining it with (\ref{eq:in-nonhom}) gives (f).
\end{proof}

\begin{proof}[Proof of (b).]
Now, to prove that ${\mathcal{D}}^h$ is a universal Gr\"obner basis for $I(\widetilde{L})$, we need to show that $\initial_< {\mathcal{D}^h}$ generates $\initial_< I(\widetilde{L}) = J(\blue{x_0}=1, y_0=1)$ for any $<$. Take a generator $m$ of $J(\blue{x_0}=1, y_0=1)$; by definition this is the initial term of a homogenized cocircuit $g_D^h$ of $I(\widetilde{L}_{hom})$ after setting $\blue{x_0}=y_0=1$; here 
$D \in \mathcal{D}_{hom}$ is a cocircuit of $M_{hom}$.

Let ${D} = \{d_1< \cdots < d_k\}$ so that $m = \blue{x_{d_1}}y_{d_2} \cdots y_{d_k}|_{\blue{x_0}=y_0=1}$.  
If $0 \in D$, then $D-0$ is a cocircuit of $M$ with initial term $m$, so $m \in  \initial_< {\mathcal{D}^h}$. If $0 \notin D$ and $D$ is also a cocircuit of $M$, then $m \in  \initial_< {\mathcal{D}^h}$ automatically. Finally, assume that $0 \notin D$ and $D$ is not a cocircuit of $M$. 
Then one may verify that there is a cocircuit $D' \subseteq D$ of $M$ containing $d_1$ (which must be its smallest element). 
Therefore $\initial_< f^h_{D'}$ divides $\blue{x_{d_1}}y_{d_2} \cdots y_{d_k}= m$ and $m \in \initial_<{\D^h}$ as desired.

Since $<$ was arbitrary, it follows that $\mathcal{D}$ is a universal Gr\"obner basis for $I(\widetilde{L})$. Again, no term in any polynomial in $\mathcal{D}$ divides another, so $\mathcal{D}$ is reduced under any term order. 
\end{proof}

\begin{proof}[Proof of (e).]
Recall from (\ref{eq:local})  that each initial ideal of $I(\widetilde{L})$ is the localization of an initial ideal of $I(\widetilde{L}_{hom})$ at $\blue{x_0}=1, y_0=1$. Since there are at most $r! \cdot b_{hom}$ such ideals, the  result follows.
%
%
\end{proof}

%

\begin{proof}[Proof of (d).] It follows from (f) that
\begin{equation}\label{eq:bideg}
\bideg \initial_<(\widetilde{I}) = \sum_{B \textrm{ basis of M}} s^{|IA^{hom}(B)|}t^{r-|IA^{hom}(B)|}
\end{equation}
where $IA^{hom}(B)$ is the set of internally active elements of $B$ as a basis of $M_{hom}$.

1. If $0$ is the largest element of $M_{hom}$ then it does not affect the internal activity of any basis. Therefore $IA'(B) = IA(B)$ for all $B$ and $\bideg \initial_<(\widetilde{I}) = t^rh_M(s/t)$ by Theorem \ref{thm:Crapo}.

2. Suppose $0$ is the smallest element of $M_{hom}$.
From Theorem \ref{th:morphism} it follows easily that
\[
T_{M \rightarrow M'}(x,y,0) = \sum_{B \textrm{ basis of M}} x^{|IA'(B)|} y^{|EA(B)|}, 
\]
so it remains to show that $IA'(B) = IA^{hom}(B)$ for any basis $B$ of $M$. We prove both inclusions.

Let $i \in IA'(B)$. Then there is a cocircuit $D \subseteq (E-B) \cup i$ of $M'$ whose smallest element is $i$. Now, every cocircuit of $M'=M_{hom}/0$ is a cocircuit of $M_{hom}$, so $i$ is the minimum in the unique cocircuit $D \subseteq (E-B) \cup i$ of $M_{hom}$. Therefore $i \in IA^{hom}(B)$.

Let $i \in IA^{hom}(B)$. Then $i$ is the minimum element in the unique cocircuit $D \subseteq (E-B) \cup i$ of $M_{hom}$. Since $0<i$, we must have $0 \notin D$. But every cocircuit of $M_{hom}$ not containing $0$ is also a cocircuit of $M_{hom}/0=M'$, and hence $i$ is minimum in the unique cocircuit $D \subseteq (E-B) \cup i$ of $M'$. Therefore $i \in IA'(B)$. The desired result follows.
\end{proof}

\begin{theorem}\label{thm:affinebettinumbers}
Let $L$ be an affine linear $d$-space in $\AA^n$, and $I(\widetilde{L})$ the ideal of its closure in $(\mathbb{P}^1)^n$. The non-zero multigraded Betti numbers of $S/I(\widetilde{L})$ 
are precisely:
\[
\beta_{i,\a} (S/I(\widetilde{L})) = 
|\mu(F, \widehat{1})|
\]
for each flat $F$ of $M$, where $i=r-r(F)$, and $\a = e_{[n]-F}$.
Here $\mu$ is the M\"obius function of the lattice of flats of $M$.

Furthermore, all of the initial ideals have the same Betti numbers: 
$$\beta_{i,\a} (S/I(\widetilde{L})) = \beta_{i,\a} (S/(\initial_< I (\widetilde{L})))$$
for all $\a$ and for every term order $<$.  
\end{theorem}

\begin{proof}
In view of Theorem \ref{affinelongthm}(f), Lemma \ref{lemma:regsequence} still applies here, and the proof of Theorem \ref{thm:bettinumbers} extends directly from the linear case to the affine case.
\end{proof}

\begin{theorem}
If $L$ is an affine subspace of $\AA^n$, the ideal $I(\widetilde{L})$ and all of its initial ideals are Cohen-Macaulay. 
\end{theorem}

\begin{proof}
The proof of Theorem \ref{CM} applies here as well.
\end{proof}

\section{\textsf{Future directions.}}
\begin{itemize}
\item
What can be said about the closure of a linear space $L \subset \AA^n$ induced by an embedding $\AA^n \hookrightarrow \PP^{a_1} \times \cdots \times \PP^{a_k}$ where $\{a_1, \ldots, a_k\}$ is a partition of $n$?
\item
Is there a common generalization of our results and the recent work of Li \cite{Li}?
\item
We believe the simplicial complex $B_<(M)$  deserves further study. What is its topology? Is it shellable? How is it related to the active order defined by Las Vergnas \cite{LV} and further studied by Blok and Sagan \cite{BlokSagan}? These questions are the subject of an upcoming project.
\item
The Tutte polynomial of a matroid can be described in terms of the interaction of the internal and external activities of the bases of $M$. In that spirit, is there a simplicial complex extending $B_<(M)$ which simultaneously involves the internal and external activities of the bases of $M$? Ideally we would like it to come from a natural geometric construction.
\item
The polynomial $T_{M \rightarrow M'}(x,1,0)$ might deserve to be called the $h$-polynomial of the matroid morphism $M \rightarrow M'$, in light of Theorem \ref{affinelongthm}(d). Does it satisfy some of the properties of the $h$-polynomial of a matroid, which has been studied extensively?
\end{itemize}

\section{\textsf{Acknowledgments.}}
We would like to thank Lauren Williams for organizing an open problem session at UC Berkeley in the Spring of 2013, where this joint project was born. The first author would also like to thank Lauren and UC Berkeley for their hospitality during the academic year 2012-2013, when part of this work was carried out.  The second author is grateful to David Eisenbud, Binglin Li, and Bernd Sturmfels for helpful conversations concerning this project.

\bibliographystyle{amsalpha}
\bibliography{LinearClosures}

\end{document}